\documentclass[twosided]{amsart}
\usepackage{amsmath}
\usepackage{amsfonts}
\usepackage{amssymb,enumerate}
\usepackage{amsthm}
\usepackage{hyperref}
\usepackage{centernot}
\usepackage{stmaryrd}
\usepackage[margin=1.2in]{geometry}

\theoremstyle{plain}
\newtheorem{theorem}{Theorem}[section]
\newtheorem{lemma}[theorem]{Lemma}
\newtheorem{definition}[theorem]{Definition}
\newtheorem{corollary}[theorem]{Corollary}
\newtheorem{proposition}[theorem]{Proposition}

\newtheorem{fact}[theorem]{Fact}
\begin{document}

\title{On Keisler Singular-Like Models II}

\author{ Shahram Mohsenipour}
\address{\ Shahram Mohsenipour,
         School of Mathematics, Institute for Research in Fundamental Sciences (IPM)\\
         P. O. Box 19395-5746, Tehran, Iran}
\email{mohseni@ipm.ir}
\thanks{This research was in part supported by a grant from IPM (No. 94030403)}

\begin{abstract}
Keisler in \cite{keislerordering} proved that if $\theta$ is a
strong limit cardinal and $\lambda$ is a singular cardinal, then the
transfer relation $\theta\longrightarrow\lambda$ holds. In a previous paper \cite{shahramsingular},
we studied initial elementary submodels of the $\lambda$-like models produced in the proof of Keisler's
transfer theorem when $\theta$ is further assumed to be regular i.e., $\theta$ is strongly inaccessible.
In this paper we deal with a much more difficult situation. Some years ago Ali Enayat asked the author whether Keisler's singular-like models
can have elementary end extensions. We give a positive answer to this question.
\end{abstract}

\maketitle
\bibliographystyle{amsplain}

\section{Introduction}
Suppose $\mathcal{L}=\{<,\ldots\}$ is any countable first order language in which $<$ is interpreted as a linear order. Let $T$ be any complete first order theory in the language $\mathcal{L}$ such that $T$ has a $\theta$-like model $M$, where $\theta$ is a strongly inaccessible cardinal. In this paper we continue our previous investigations of model theory of $T$ by showing that for any singular cardinal $\lambda$, there is a class of Keisler $\lambda$-like models of $T$ such that each model in the class has arbitrary large elementary end extensions. We fix $\mathcal{L}$, $T$, $\theta$, $\lambda$ and $M$ as above. Now to state the result more precisely, we add new function symbols to $\mathcal{L}$ as Skolem functions and show the resulting language by $\mathcal{L}^{S}$. Also let $T_{\mathrm{skolem}}$ be the usual Skolem theory asserting that ``there are Skolem functions". Suppose $\mathcal{L}^{S}(C)=\mathcal{L}^{S}\cup C$, where $C=\{c_{ij}|i<\eta,j <\mu_{i}\}$ in which $\eta=cf(\lambda)$ and $\langle\mu_{i};i<\eta\rangle$ is an increasing sequence of cardinals with $\lim_{i<\eta} \mu_{i}=\lambda$. Keisler in \cite{keislerordering} introduced an $\mathcal{L}^{S}(C)$-theory $\Sigma\supset T_{\mathrm{skolem}}$ such that

\begin{theorem}[Keisler\cite{keislerordering}]\label{keislermain}

$\mathrm{(i)}$ $T+\Sigma$ is consistent,

$\mathrm{(ii)}$ for any model $N\models\Sigma$, the elementary submodel $N^{*}\prec N$ generated by $C$ under the Skolem functions is $\lambda$-like.
\end{theorem}

Now we call the singular-like model $N^{*}$ in the above theorem a {\em Keisler model} of $T$. In order to prove the much harder part (i) of Theorem \ref{keislermain}, namely the consistency of $T+\Sigma$, Keisler defined his {\em Large Sets} which are special ``large" sets whose members are finite matrices with elements coming from the initial model $M$ and then by using Erd\"{o}s-Rado's polarized partition theorem he proved some combinatorial properties of the large sets. Let $\Sigma^{'}$ be a finite part of $\Sigma$, then it was shown that there is a large set whose every element can interpret the finitely many $c_{ij}$'s appearing in $\Sigma^{'}$ in such a way that $\Sigma^{'}$ holds in $M$. Therefore $T+\Sigma$ is consistent.

Now we are in the position to state our result. We shall introduce an $\mathcal{L}^{S}(C)$-theory $\Sigma_{1}\supset T_{\mathrm{skolem}}$ such that

\begin{theorem}\label{maintheorem}

$\mathrm{(i)}$ $T+\Sigma+\Sigma_{1}$ is consistent,

$\mathrm{(ii)}$ for any model $N\models\Sigma+\Sigma_{1}$, the elementary submodel $N^{*}\prec N$ generated by $C$ under the Skolem functions is $\lambda$-like and has elementary end extensions of any cardinality $\geq\lambda$.
\end{theorem}

Now we see that any $N^{*}\models T$ generated by Theorem \ref{maintheorem} is a Keisler singular-like model with arbitrary large elementary end extensions. Thus Enayat's question is answered positively. To prove Theorem \ref{maintheorem} we follow the same strategy but it seems our theorem can not be resolved in the framework of Keisler's Large Sets, so we are forced to work with more general sets which we call {\em Superlarge Sets} in order to locally interpret the axioms of $\Sigma_{1}$ in $M$. In fact the main technical parts of this work are the proofs of two combinatorial properties of superlarge sets in Propositions \ref{proposition1} and  \ref{proposition3}. At the end of the paper by giving an example we show that in Theorem \ref{maintheorem}, the strong inaccessibility of $\theta$ is necessary.

\section{Proof of the second part of Theorem \ref{maintheorem}}

We begin this section by reviewing some partition theorems of Erd\"{o}s and Rado for infinite cardinals which as in the case of Keisler's large sets will be used to demonstrate some combinatorial properties of superlarge sets.  Let $\kappa$ be a cardinal, we denote by $[X]^{\kappa}$ the set of all subsets of $X$ of cardinality $\kappa$. Note that if $X$ is a linearly ordered set and $r$ is a positive integer, we identify $[X]^{r}$ by the set of all increasing sequences of length $r$ coming from $X$.

\begin{theorem}[Erd\"{o}s and Rado] \label{erdosradop}
For any infinite cardinal $\kappa$ and any $r<\omega$
\[
\beth_{r}(\kappa)^{+}\longrightarrow(\kappa^{+})^{r+1}_{\kappa}.
\]
\end{theorem}
We also recall Erd\"{o}s and Rado's {\em polarized} partition relation. Let $r$, $s$ be positive integers and $\mu$, $\kappa_{i},\lambda_{i}$ for $1\leq i\leq s$ be cardinals (finite or infinite). The expression
\[
(\kappa_{1},\dots,\kappa_{s})\longrightarrow(\lambda_{1},\dots,\lambda_{s})^{r}_{\mu}
\]
means that for any partition of the set
\[
[\kappa_{1}]^{r}\times\dots\times[\kappa_{s}]^{r}
\]
into $\mu$ parts, there exist sets
\[
X_{1}\in[\kappa_{1}]^{\lambda_{1}},\dots,X_{s}\in[\kappa_{s}]^{\lambda_{s}}
\]
such that the set
\[
[X_{1}]^{r}\times\dots\times[X_{s}]^{r}
\]
lies entirely within one part of the definition.
\begin{theorem}[Erd\"{o}s and Rado]\label{erdosradopp}
Suppose $\kappa_{i},\lambda_{i}$ are infinite cardinals for $1\leq i\leq s+t$ such that
\[
(\kappa_{1},\dots,\kappa_{s})\longrightarrow(\lambda_{1},\dots,\lambda_{s})^{r}_{\mu}
\]
and
\[
(\kappa_{s+1},\dots,\kappa_{s+t})\longrightarrow(\lambda_{s+1},\dots,\lambda_{t})^{r}_{\mu^{'}}
\]
where $\mu^{'}\geq\mu^{\kappa_{1}.\dots.\kappa_{s}}$. Then
\[
(\kappa_{1},\dots,\kappa_{s+t})\longrightarrow(\lambda_{1},\dots,\lambda_{s+t})^{r}_{\mu}.
\]
\end{theorem}

The following corollary of Erd\"{o}s-Rado's polarized partition theorem will be very useful.

\begin{corollary}\label{usefulcorollary}
Suppose that for $1\leq i\leq s$, $\kappa_{i},\lambda_{i}$ are infinite cardinals and
\[
\kappa_{i}>\beth_{r-1}(\lambda_{i}),\,\,\,\,\,\lambda_{i+1}\geq 2^{\kappa_{i}}.
\]
Then
\[
(\kappa_{1},\dots,\kappa_{s})\longrightarrow(\lambda_{1}^{+},\dots,\lambda_{s}^{+})^{r}_{\lambda_{1}}.
\]
\end{corollary}
\begin{proof}
By Theorem \ref{erdosradop} we have
\[
\kappa_{i}\longrightarrow(\lambda_{i}^{+})^{r}_{\lambda_{i}},\,\,\,\,\, 1\leq i\leq s.
\]
Also
\[
\lambda_{i+1}\geq 2^{\kappa_{i}}=\kappa_{i}^{\kappa_{i}}\geq\lambda_{1}^{\kappa_{1}.\dots.\kappa_{i}}.
\]
The corollary now follows from Theorem \ref{erdosradopp} by induction on $i$.
\end{proof}

Now we fix our notations from the previous section. Suppose $\mathcal{L}=\{<,\dots\}$ is any countable first order language in which $<$ is always interpreted as a linear ordering and $T$ is an $\mathcal{L}$-theory such that $T$ has a $\theta$-like model $M$ where $\theta$ is a strongly inaccessible cardinal. Let $\mathcal{L}^{S}$ be the result of adding Skolem functions to $\mathcal{L}$ and $T_{\mathrm{skolem}}$ be the usual Skolem theory. Obviously $M$ can be expanded to be a model of $T_{\mathrm{skolem}}$. Also let $\mathcal{L}^{S}(C)$ be the language produced by adding a set of doubly indexed constants $C=\{c_{ij}|i<\eta,j <\mu_{i}\}$ to $\mathcal{L}^{S}$ where $\lambda$ is a singular cardinal, $\eta=cf(\lambda)$ and $\langle\mu_{i};i<\eta\rangle$ is an increasing sequence of cardinals with $\lim_{i<\eta} \mu_{i}=\lambda$. Since $\theta$ is strongly inaccessible, by an easy Skolem Hull argument we can write $M$ as the union of an elementary end extension chain of its $\mathcal{L}^{S}$-submodels: $M=\bigcup_{i<\theta} M_{i}$ such that for any limit ordinal $\sigma <\theta$, we have $M_{\sigma}=\bigcup_{i<\sigma} M_{i}$. Now we define a function $F\colon M\longrightarrow\theta$ such that for any $a\in M$, $F(a)$ is the least ordinal $i<\theta$ with $a\in M_{i}$. Obviously $F(x)$ is always a successor ordinal $<\theta$. We frequently use this simple implication of the definition of $F$ that if $\tau(x_{1},\dots, x_{n})\in \mathcal{L}^{S}$ is a term and $\{a_{1},\dots,a_{n}, b\}\subset M$ such that $F(b)>\max(F(a_{1}),\dots,F(a_{n}))$, then $\tau(a_{1},\dots, a_{n})<b$. Suppose $r,s$
are two positive integers. We consider sequences $\textbf{x}$ of length $s$,
each term being a sequence of length $r$. For such sequences we write
\begin{center}
$\textbf{x}=\langle
\textbf{x}_{1},\dots,\textbf{x}_{s}\rangle=\big{\langle}\langle x_{11},\dots,x_{1r}\rangle,\ldots,
\langle x_{s1},\ldots,x_{sr}\rangle\big{\rangle}$.
\end{center}
Sometimes we denote $i$th coordinate $x_{i}$ of any tuple $\textbf{x}=\langle x_{1},\dots,x_{n}\rangle$ by $\textbf{x}(i)$ for $1\leq i\leq n$. We define $[F]^{r,s}$ to be the set of all $s$-tuples $\textbf{x}$ of elements
of $[M]^{r}$(the set of all increasing $r$-sequences of $M$) such that
\begin{center}
$F(x_{ij})=F(x_{il}),$ \,  $i=1,\ldots,s$\,\,\, and \,\,\,
$j,l=1,\dots,r.$
\end{center}
and
\begin{center}
$F(x_{11})<F(x_{21})<\ldots<F(x_{s1}).$
\end{center}
Then
\begin{center}
$[F]^{r,s}=\bigcup\{[F^{-1}(\alpha_{1})]^{r}\times\ldots\times
[F^{-1}(\alpha_{s})]^{r};\alpha_{1}<\ldots<\alpha_{s}<\theta\}.$
\end{center}
Suppose $A\subset M$, we use $[F|A]^{rs}$ to denote the set $\{\mathbf{x}\in [F]^{rs}\,|\,x_{ij}\in A\}$. We use a game theoretical language to introduce superlarge sets. For each positive integer $e\leq s$ and a subset $S\subset[F]^{r,s}$, we consider
a game $G(S,e)$ between two players I and II. In this game each
player has $e$ moves. Put $f=s-e$. Player I moves first, and for his first move
he chooses a cardinal $\mu_{1}<\theta$. Then II chooses an ordinal
$\beta_{1}<\theta$. Then I chooses a cardinal $\mu_{2}<\theta$ and then II
chooses an ordinal $\beta_{2}<\theta$, and so on until the player I chooses a cardinal $\mu_{e}$ for his last move. The player II for his last move will choose a sequence of ordinals $\langle\beta_{e+i}|i<\theta\rangle$ of length $\theta$. We say that the player II {\em wins}
the game $G(S,e)$ if
\begin{center}
$\beta_{1}<\beta_{2}<\dots<\beta_{e}<\dots<\beta_{e+i}<\cdots$ for $i<\theta$
\end{center}
and there exist sets
\begin{center}
$X_{1}\in[F^{-1}(\beta_{1})]^{\mu_{1}},\dots,X_{e}\in[F^{-1}
(\beta_{e})]^{\mu_{e}}$
\end{center}
as well as sets
\begin{center}
$X_{e+i}\subset F^{-1}(\beta_{e+i})$ for $1\leq i<\theta$
\end{center}
such that
\begin{center}
$\sup\bigr\{|X_{e+i}|;i<\theta\bigr\}=\theta$
\end{center}
where $|X|$ denotes the cardinality of $X$ and
\begin{center}
$\displaystyle{\prod_{1\leq i\leq e}}\big{[}X_{i}\big{]}^{r}\times \,\,\,\big{[}F|(\!\!\displaystyle{\bigcup_{1\leq i<\theta}}X_{e+i})\big{]}^{r,f}\subset S.$
\end{center}
Otherwise I wins. Note that if $f=0$, then the right hand set of the above product is empty. Since $e$ is finite, it is clear that exactly one player has a
winning strategy for the game $G(S,e)$.
\begin{definition}
We say that a set $S\subset [F]^{rs}$ is \textsl{e-superlarge} $(1\leq e\leq s)$ if the player II has a winning strategy for the game $G(S,e)$.
\end{definition}
It is trivial that any $e$-superlarge subset of $[F]^{r,s}$ is nonempty.
\begin{definition}
Let $\Sigma_{1}$ be the following $\mathcal{L}^{S}(C)$-theory:
\begin{itemize}
\item[(i)] $T_{\mathrm{skolem}}$ plus the axioms for $<$ to be a linear order.

\item[(ii)] $c_{ij}<c_{kl}$ iff $(i,j)<(k,l)$ in the
lexicographical order.

\item[(iii)] $\tau(c_{i_{1}j_{1}},\ldots,c_{i_{n}j_{n}})<c_{ij}$, where
$\tau(v_{1},\dots,v_{n})$ is a term of $\mathcal{L}^{S}$,
$i_{1},\ldots,i_{n}<i$ and $j,j_{1},\dots,j_{n}$ are arbitrary ordinals.

\item[(iv)] If  $i_n>1$ and $\tau(v_{1},\dots,v_{n})$ is a term of $\mathcal{L}^{S}$ and $\tau(c_{i_{1}j_{1}},\ldots,c_{i_{n}j_{n}})<c_{(i_{n}-1)j}$, then
\begin{center}
$\tau(\overline{c},c_{i_{q+1}j_{q+1}},\ldots,c_{i_{n}j_{n}})=\tau(\overline{c},c_{ul_{1}},\ldots,c_{ul_{n-q}})$,
\end{center}
where $u\geq i_{n}$, $q$ is the greatest integer such that $i_{q}\neq i_{n}$ and $l_{1},\dots,l_{n-q}$ are arbitrary ordinals and $\overline{c}=\langle c_{i_{1}j_{1}},\ldots,c_{i_{q}j_{q}}\rangle$. If there is no such $q$, namely $i_{1}=\dots=i_{n}$, then obviously the above equality becomes:
\[
\tau(c_{i_{1}j_{1}},\ldots,c_{i_{n}j_{n}})=\tau(c_{ul_{1}},\ldots,c_{ul_{n}}).
\]
We add that in the above axioms we suppose that in any expression of terms with constants such as $\tau(c_{m_{1}n_{1}},\ldots,c_{m_{k}n_{k}})$, the sequence $\langle c_{m_{1}n_{1}},$ $\ldots,c_{m_{k}n_{k}}\rangle$ is increasing.
\end{itemize}
\end{definition}

Now we prove the part (ii) of Theorem \ref{maintheorem}. Notationally we will make no difference between the symbols of the language and their interpretations

\begin{proof}[Proof of Theorem \ref{maintheorem} (ii)]
Let $N$ be a model of $\Sigma+\Sigma_{1}$ and $N^{*}\prec N$ be generated by $C$ under the Skolem functions and $\kappa$ be any infinite cardinal. Let $D=\{d_{i}|i<\kappa\}$ be a set of new constant symbols which we add to the language $\mathcal{L}^{S}(C)$ and denote the resulting language by $\mathcal{L}^{S}(C\cup D)$. We introduce a set of axioms $\Pi$ in $\mathcal{L}^{S}(C\cup D)$ and show that (i) $\Pi$ is consistent with $Th(N^{*},\mathcal{L}^{S}(C))$ (ii) for any model $K$ of $\Pi + Th(N^{*},\mathcal{L}^{S}(C))$, if $K^{*}\prec K$ is generated by $C\cup D$ then we have $N^{*}\prec_{eee} K^{*}$. Let $\Pi$ be the following $\mathcal{L}^{S}(C\cup D)$-theory:
\begin{itemize}
\item[(i)] $d_{i}<d_{j}$ \textit{iff} $i<j$.

\item[(ii)] $d_{0}>c_{ij}$ \textit{for any} $i,j$.

\noindent If $\tau(c_{i_{1}j_{1}},\dots,c_{i_{n}j_{n}},c_{(i+1)0},\dots,c_{(i+1)m})<c_{ij}$ for some $i>i_{n}$ and $j$, then
\item[(iii)] \textit{for any
increasing sequence} $\langle d_{l_{0}},\dots,d_{l_{m}}\rangle$:
\begin{center}
$\tau(\overline{c},d_{l_{0}},\dots,d_{l_{m}})=\tau(\overline{c},c_{(i+1)0},\dots,c_{(i+1)m})$,
\end{center}
\textit{where} $\overline{c}=\langle c_{i_{1}j_{1}},\dots,c_{i_{n}j_{n}}\rangle.$

\noindent If for any $i>i_{n}$ and $j$, $\tau(c_{i_{1}j_{1}},\dots,c_{i_{n}j_{n}},c_{(i+1)0},\dots,c_{(i+1)m})>c_{ij}$, then
\item[(iv)] \textit{for any
increasing sequence} $\langle d_{l_{0}},\dots,d_{l_{m}}\rangle$:
\begin{center}
$\tau(c_{i_{1}j_{1}},\dots,c_{i_{n}j_{n}},d_{l_{0}},\dots,d_{l_{m}})>c_{ij}$, \,\,\,\textit{for any} $j$.\\
\end{center}
\end{itemize}

To prove the consistency of $\Pi+Th(N^{*},\mathcal{L}^{S}(C))$, we assume that $\Pi^{'}$ is a finite part of $\Pi$. We show that $N^{*}$ is a model of $\Pi^{'}$ via interpreting the finitely many constant symbols $d_{i}$'s appearing in $\Pi^{'}$ by some suitable $c_{ij}$'s. Let $c_{i_{1}j_{1}},\dots,c_{i_{n}j_{n}}$ be all the elements of $C$ which appeared in $\Pi^{'}$ where $i_{1}\leq\dots\leq i_{n}$. Also suppose $d_{l_{0}},\dots,d_{l_{m}}$ are all the constant symbols from $D$ appearing in $\Pi^{'}$. Now we interpret $d_{l_{0}},\dots,d_{l_{m}}$ by $c_{(i_{n}+1)0},\dots,c_{(i_{n}+1)m}$ in $N^{*}$, respectively. We also interpret all the Skolem terms and all $c_{ij}$'s in $N^{*}$ as in the previous. It is evident $\Sigma_{1}$(ii) will guarantee that all sentences of types of $\Pi$(i) and $\Pi$(ii) occurring in $\Pi^{'}$ hold in $N^{*}$. It remains to show how the above interpretation of $\Pi^{'}$ makes those sentences of types $\Pi$(iii) and $\Pi$(iv) true in $N^{*}$. Consider a sentence of type $\Pi$(iii), say,
\begin{equation}\label{eq1}
\tau(\overline{c},d_{k_{0}},\dots,d_{k_{q}})=\tau(\overline{c},c_{(a+1)0},\dots,c_{(a+1)q}),
\end{equation}
where $\overline{c}=\langle c_{a_{1}b_{1}},\dots,c_{a_{p}b_{p}}\rangle$ and
\begin{center}
$\{c_{a_{1}b_{1}},\dots,c_{a_{p}b_{p}}\}\cup\{c_{(a+1)0},\dots,c_{(a+1)q}\}\subset\{c_{i_{1}j_{1}},\dots,c_{i_{n}j_{n}}\}$,
\end{center}
as well as
$\{d_{k_{0}},\dots,d_{k_{q}}\}\subset\{d_{l_{0}},\dots,d_{l_{m}}\}$.
Since the sentence (\ref{eq1}) is in $\Pi^{'}$, we can deduce that it must already happened that
$\tau(\overline{c},c_{(a+1)0},\dots,c_{(a+1)q})<c_{aj}$,for some $j$.
Then by recalling that $a\leq i_{n}$, $\Sigma_{1}$(iv) would imply that
\begin{center}
$\tau(\overline{c},c_{(a+1)0},\dots,c_{(a+1)q})=
\tau(\overline{c},c_{(i_{n}+1)e_{0}},\dots,c_{(i_{n}+1)e_{q}})$
\end{center}
for any $e_{0},\dots,e_{q}$, in particular when $e_{i}$'s are such that $l_{e_{0}}=k_{0},\dots,l_{e_{q}}=k_{q}$. So $c_{(i_{n}+1)e_{0}}$$,\dots,$$c_{(i_{n}+1)e_{q}}$ interpret $d_{k_{0}},\dots,d_{k_{q}}$, respectively in such a way that $N^{*}$ satisfies the sentence (\ref{eq1}). Similarly consider a sentence of type $\Pi$(iv): fix  $i_{*},j_{*}$ such that
\begin{equation}\label{eq2}
\tau(\overline{c},d_{k_{0}},\dots,d_{k_{q}})>c_{i_{*}j_{*}}.
\end{equation}
According to $\Pi$(iv), it must already happened that for all $j$:
\begin{equation}\label{eq3}
\tau(\overline{c},c_{(i_{*}+1)0},\dots,c_{(i_{*}+1)q})>c_{i_{*}j}.
\end{equation}
We claim that for any $e_{0},\dots,e_{q}$ and for all $j$:
\begin{center}
$\tau(\overline{c},c_{(i_{n}+1)e_{0}},\dots,c_{(i_{n}+1)e_{q}})>c_{i_{*}j}$.
\end{center}
If not, then there are $j^{*}$ and $e^{*}_{0},\dots,e^{*}_{q}$ such that
\[
\tau(\overline{c},c_{(i_{n}+1)e^{*}_0},\dots,c_{(i_{n}+1)e^{*}_q})<c_{i_{*}j^{*}},
\]
but $i_{*}\leq i_{n}$ and in this case, $\Sigma_{1}$(iv) implies that
\begin{center}
$\tau(\overline{c},c_{(i_{*}+1)0},\dots,c_{(i_{*}+1)q})\,=\,
\tau(\overline{c},c_{(i_{n}+1)e^{*}_0},\dots,c_{(i_{n}+1)e^{*}_q})$
\end{center}
therefore $\tau(\overline{c},c_{(i_{*}+1)0},\dots,c_{(i_{*}+1)q})<c_{i_{*}j^{*}}$,
which contradicts the inequality (\ref{eq3}), so we have proved the claim. Again, if $e_{i}$'s are such that $l_{e_{0}}=k_{0},\dots,l_{e_{q}}=k_{q}$, then $c_{(i_{n}+1)e_{0}},$$\dots,c_{(i_{n}+1)e_{q}}$ do interpret $d_{k_{0}},\dots,d_{k_{q}}$, respectively in such a way that $N^{*}$ satisfies the sentence (\ref{eq2}). This completes the proof of (i), namely, $\Pi$ is consistent with $Th(N^{*},\mathcal{L}^{S}(C))$. To demonstrate (ii), let $K$ be a model of $\Pi + Th(N^{*},\mathcal{L}^{S}(C))$ and let $K^{*}\prec K$ be generated by $C\cup D$. Obviously we can identify with $N^{*}$, that elementary submodel of $K^{*}$ which is generated by $C$. We must show that $N^{*}\prec_{eee} K^{*}$. Consider a typical element $\tau(c_{u_{1}v_{1}},\dots,c_{u_{n}v_{n}},d_{l_{0}},\dots,d_{l_{m}})$ of $K^{*}$. For the sake of brevity we write $\overline{c_{uv}}=\langle c_{u_{1}v_{1}},\dots,c_{u_{n}v_{n}}\rangle$. It suffices to show:
\begin{center}
either $\tau(\overline{c_{uv}},d_{l_{0}},\dots,d_{l_{m}})>N^{*}$ or
$\tau(\overline{c_{uv}},d_{l_{0}},\dots,d_{l_{m}})$ $\in N^{*}$.
\end{center}
There are two separate cases: Case (I): for any $u_{n}<u$ and $v$:
\begin{center}
$\tau(\overline{c_{uv}},c_{(u+1)0},\dots,c_{(u+1)m})>c_{uv}$.
\end{center}
Case (II): for some $u_{n}\leq u_{*}$ and $v_{*}$:
\begin{center}
$\tau(\overline{c_{uv}},c_{(u_{*}+1)0},\dots,c_{(u_{*}+1)m})<c_{u_{*}v_{*}}$.
\end{center}
If Case (I) occurs then by $\Pi$(iv) we have for any $u$ and $v$:
\begin{center}
$\tau(\overline{c_{uv}},d_{l_{0}},\dots,d_{l_{m}})>c_{uv}$.
\end{center}
Since $c_{uv}$'s are cofinal in $N^{*}$, this means that
\begin{center}
$\tau(\overline{c_{uv}},d_{l_{0}},\dots,d_{l_{m}})>N^{*}$.
\end{center}
If Case (II) occurs, then $\Pi$(iii) implies that
\begin{center}
$\tau(\overline{c_{uv}},d_{l_{0}},\dots,d_{l_{m}})=
\tau(\overline{c_{uv}},c_{(u_{*}+1)0},\dots,c_{(u_{*}+1)m})$,
\end{center}
which means that
\begin{center}
$\tau(\overline{c_{uv}},d_{l_{0}},\dots,d_{l_{m}})\in N^{*}$.
\end{center}
Therefore the proof of $N^{*}\prec_{eee} K^{*}$ and consequently the proof of the part (ii) of Theorem \ref{maintheorem} is complete.
\end{proof}

\section{First combinatorial property of superlarge sets}

We first note that the set $\Sigma_{1}$ defined in the previous section is ``homogenous" in the sense of Keisler. We call two strictly increasing sequences
\[
\langle c_{i_{1}j_{1}},\dots,c_{i_{n}j_{n}}\rangle,\,\,\,\,\,\langle c_{k_{1}l_{1}},\dots,c_{k_{n}l_{n}}\rangle
\]
{\em similar} iff
\[
i_{p}=i_{q}\,\,\, \mathrm{iff}\,\,\, k_{p}=k_{q},\,\,\,\,\, p,q=1,\dots,n.
\]
Then whenever $\Sigma_{1}$ contains a sentence $\sigma$, it also contains every sentence formed by replacing the sequence of all constants occurring in $\sigma$ by a similar sequence of constants.

Let $C^{*}=\{c_{ij}|i,j<\omega\}$ and let $\Sigma_{1}^{*}$ be the $\mathcal{L}^{S}(C^{*})$-theory such that its sentences are exactly the sentences of $\Sigma_{1}$ except that this time the constants $c_{ij}$'s come from the set $C^{*}$. By homogeneouity, it is easy to see that

\begin{lemma}\label{sigmainfinite} For any $\mathcal{L}^{S}$-theory $\Gamma$, $\Gamma+\Sigma_{1}$ is consistent iff $\Gamma+\Sigma_{1}^{*}$ is consistent.
\end{lemma}

We now move towards proving the combinatorial Propositions \ref{proposition1} which is one of our main tools to prove part (i) of Theorem \ref{maintheorem}. First of all we introduce an important notation in this paper. Suppose $\sigma$ is a sentence of the language $\mathcal{L}^{S}(C^{*})$ and let $r,s$ be large enough positive integers so that for any $c_{ij}$ occurring in $\sigma$, we have $i\leq s$ and $j\leq r$. Let $\textbf{a} \in [F]^{rs}$, namely
\begin{center}
$\textbf{a}=\big{\langle}\langle a_{11},\dots,a_{1r}\rangle,\ldots,
\langle a_{s1},\ldots,a_{sr}\rangle\big{\rangle}$.
\end{center}
By $M\models\sigma(\textbf{a})$, we mean that the sentence $\sigma$ holds in the model $M$, when we substitute any $c_{ij}$ occurring in $\sigma$ by $a_{ij}$. Similarly let $\tau(c_{i_{1}j_{1}},\ldots,c_{i_{n}j_{n}})$ be a term with constants such that $i_{n}\leq s$ and $\max{\{j_{1},\dots,j_{n}\}}\leq r$, we write $\tau(\textbf{a})$ as an abbreviation for $\tau(a_{i_{1}j_{1}},\ldots,a_{i_{n}j_{n}}).$ Obviously this may cause an ambiguity. For example if $\tau(c_{i_{1}j_{1}},\dots,c_{i_{n}j_{n}})$ and $\tau(c_{k_{1}l_{1}},\dots,c_{k_{n}l_{n}})$ are two terms with constants such that $i_{n},k_{n}\leq s$ and $\max{\{j_{1},\dots,j_{n},l_{1},\dots,l_{n}\}}\leq r$, then $\tau(\textbf{a})$ may have two different values. Similar ambiguities may arise also when we deal with $\sigma(\mathbf{a})$, so to avoid such situations, whenever we talk about $\tau(\textbf{a})$ and $\sigma(\mathbf{a})$ everywhere in this paper, we previously determine which set of constants is meant.

It is also useful to consider an \textit{equivalence} relation between tuples of the doubly indexed constants $c_{ij}$ which is a stronger notion than similarity. We call two strictly increasing sequences
\[
\langle c_{i_{1}j_{1}},\dots,c_{i_{n}j_{n}}\rangle,\,\,\,\,\,\langle c_{k_{1}l_{1}},\dots,c_{k_{n}l_{n}}\rangle
\]
{\em equivalent} iff
\[
i_{p}=k_{p}\,\,\, \mathrm{for}\,\,\,\,\, p=1,\dots,n.
\]
Related to the equivalent tuples of constants, we formulate a simple combinatorial Lemma \ref{combinatoriallemma} which will be very useful to organize our arguments in Propositions \ref{proposition1}, \ref{proposition2} in this section and also Proposition \ref{proposition3} in the next section. But before stating it we need to prove a fact about infinite linear orders:
\begin{fact}\label{fact}
Suppose $\langle X,<\rangle$ is an infinite linear ordering. Then for any positive integer $r$, there is $Y\subset X$ such that $|Y|=|X|$ and for any $y_{1}<y_{2}$ in $Y$ there are at least $r$ elements $x_{1}^{(i)},\dots,x_{r}^{(i)}$ $(i=1,2,3)$ in $X$ such that
\[
x_{1}^{(1)},\dots,x_{r}^{(1)}<y_{1}<x_{1}^{(2)},\dots,x_{r}^{(2)}<y_{2}<x_{1}^{(3)},\dots,x_{r}^{(3)}.
\]
We denote the set of all such $Y$ by $X^{\bullet\bullet}$.
\end{fact}
\begin{proof}
There are two cases: (i) First suppose $X$ is countable, then it is easily seen that there is an $\omega$-sequence of elements of $X$, $\langle x_{0},\dots,x_{i},\dots\rangle$ for $i<\omega$ which is either strictly increasing or strictly decreasing. So define $y_{0}=x_{0},y_{1}=x_{r+1},\dots, y_{i}=x_{ir+i}$ for $i<\omega$. Then $Y=\{y_{i};i>0\}$ will be as required. (ii) Now suppose $X$ is uncountable. Let $\sim$ be an equivalence relation on $X$ such that $x_{1}\sim x_{2}$ iff there are only finitely many elements of $X$ between $x_{1},x_{2}$. Since $X$ is uncountable, $|X/\sim|=|X|$. Now suppose $Z$ is any subset of $X$ which intersects any equivalence class of $X/\sim$ in exactly one element. Remove from $Z$ its maximum and minimum elements (if there are such elements) and call the new set $Y$ (if not, set $Y=Z$). Now it is easily seen that $Y$ satisfies the condition. In fact between any two elements of $Z$ there are infinitely many elements of $X$. \footnote{I thank Fran\c{c}ois Dorais for giving the proof of the uncountable case in response to my Mathoverflow question.}\end{proof}

\begin{lemma}\label{combinatoriallemma}
Let $\sigma$ be a $\mathcal{L}^{S}(C^{*})$-sentence with parameters and $ c_{i_{1}j_{1}},\dots,c_{i_{n}j_{n}}$ be all constant symbols occurring in $\sigma$ and they are arranged in the increasing order. Assume that $r,s$ are two positive integers such that $i_{n}\leq s$ and $j_{1},\dots,j_{n}\leq r$ and $\kappa_{1},\dots,\kappa_{s}$ are given infinite cardinals. Also suppose that there are ordinals $\beta_{1}<\dots<\beta_{s}<\theta$ together with subsets:
\[
X_{1}\in[F^{-1}(\beta_{1})]^{\kappa_{1}},\dots,X_{s}\in[F^{-1}(\beta_{s})]^{\kappa_{s}},
\]
such that far all $\mathbf{a}\in [X_{1}]^{r}\times\dots\times[X_{s}]^{r}$ we have $M\models\sigma(\mathbf{a})$ or more precisely $M\models\sigma(a_{i_{1}j_{1}},\dots,a_{i_{n}j_{n}})$. \underline{Then} there are subsets
\[
Y_{1}\subset X_{1},\dots,Y_{s}\subset X_{s},\,\,\,
\,\,\,
|Y_{1}|=\kappa_{1},\dots,|Y_{s}|=\kappa_{s}
\]
such that for all $\mathbf{a}\in [Y_{1}]^{r}\times\dots\times[Y_{s}]^{r}$ we have $M\models\sigma(a_{k_{1}l_{1}},\dots,a_{k_{n}l_{n}})$ when
\[
\langle c_{i_{1}j_{1}},\dots,c_{i_{n}j_{n}}\rangle,\,\,\,\,\,\langle c_{k_{1}l_{1}},\dots,c_{k_{n}l_{n}}\rangle
\]
are equivalent and $l_{1},\dots,l_{n}\leq r$.
\end{lemma}
\begin{proof}
According to Fact \ref{fact}, let
\[
Y_{1}\in X_{1}^{\bullet\bullet},\dots,Y_{s}\in X_{s}^{\bullet\bullet}
\]
for $i=1,\dots,s$. Now this gives us the possibility that for any $\mathbf{a}\in [Y_{1}]^{r}\times\dots\times[Y_{s}]^{r}$ we can choose a $\mathbf{b}\in [X_{1}]^{r}\times\dots\times[X_{s}]^{r}$ such that
\[
\langle b_{i_{1}j_{1}},\dots,b_{i_{n}j_{n}}\rangle=\langle a_{k_{1}l_{1}},\dots,a_{k_{n}l_{n}}\rangle.
\]
Now by the hypothesis we have $M\models\sigma(b_{i_{1}j_{1}},\dots,b_{i_{n}j_{n}})$, hence the above equality implies that $M\models\sigma(a_{k_{1}l_{1}},\dots,a_{k_{n}l_{n}})$ which proves the lemma.
\end{proof}

Now suppose $\sigma$ is a sentence of type $\Sigma_{1}^{*}$(iv). In order to state our proposition we need to keep track of the index $i_{n}$ occurring in $\sigma$ in the course of the proof, so for the sake of the easy readability, we denote it by the function $\iota(\sigma)=i_{n}$.

\begin{proposition}\label{proposition1}
Let $S\subset [F]^{r,s}$ be an e-superlarge set $(e<s)$. Suppose $\sigma$ is a sentence of type $\Sigma_{1}^{*}(iv)$ so that for all $c_{ij}$ occurring in $\sigma$ we have $i\leq s$ and $j\leq r$ and $\iota(\sigma)=e^{'}>e$. Then there is an $e^{'}$-superlarge set $S^{'}\subset S$ such that for any $\mathbf{a}\in S^{'}$ we have $M\models\sigma(\mathbf{a})$.
\end{proposition}
\begin{proof}
Suppose $\tau(c_{i_{1}j_{1}},\dots,c_{i_{q}j_{q}},\dots,c_{i_{n}j_{n}})$ and $q$ are as in the item (iv) of $\Sigma_{1}^{*}$. Set
\[
S^{'}=\{\textbf{a}\in S \, | \, M\models\sigma(\textbf{a})\}.
\]
We show that $S^{'}$ is $e^{'}$-superlarge. This will be done if we find a winning strategy:
\begin{center}
$\beta_{1}(\mu_{1}),\dots,\beta_{e^{'}}(\mu_{1},\dots,\mu_{e^{'}}),\dots,\beta_{e^{'}+i}(\mu_{1},\dots,\mu_{e^{'}}),\dots$, \,\,\,\, $i<\theta$,
\end{center}
for the player II in the game $G(S^{'},e^{'})$. Suppose the player I plays with a strategy
\[
\mu_{1},\mu_{2}(\beta_{1}),\dots,\mu_{e^{'}}(\beta_{1},\dots,\beta_{e^{'}-1}).
\]
So our task is finding $\beta_{i}$ such that guarantee the win of the player II. Since $S$ is $e$-superlarge, then the player II has a winning strategy for the game $G(S,e)$:
\begin{center}
$\gamma_{1}(\mu_{1}),\dots,\gamma_{e}(\mu_{1},\dots,\mu_{e}),\dots,\gamma_{e+i}(\mu_{1},\dots,\mu_{e}),\dots$, \,\,\,\, $i<\theta$,
\end{center}
so that $\gamma_{1}<\gamma_{2}<\dots<\gamma_{i}<\cdots$ for $1\leq i<\theta$ and
there exist sets
\begin{equation}\label{X}
X_{1}\in[F^{-1}(\gamma_{1})]^{\mu_{1}},\dots,X_{e}\in[F^{-1}
(\gamma_{e})]^{\mu_{e}}
\end{equation}\label{subset}
as well as the following sets for $1\leq i<\theta$:
\begin{equation}
X_{e+i}\subset F^{-1}(\gamma_{e+i}),
\end{equation}
such that
\begin{equation}\label{theta}
\sup\bigr\{|X_{e+i}|;i<\theta\bigr\}=\theta
\end{equation}
and
\begin{equation}\label{big}
\displaystyle{\prod_{1\leq i\leq e}}\big{[}X_{i}\big{]}^{r}
\times \,\,\,\big{[}F|(\!\!\displaystyle{\bigcup_{e<i<\theta}}X_{i})\big{]}^{r,f}\subset S,
\end{equation}
where $f=s-e$.

Now assume that in the game $G(S^{'},e^{'})$, the player II for his first $e$ moves, plays according to his winning strategy in the game $G(S,e)$. More precisely:
\begin{center}
$\beta_{j}(\mu_{1},\dots,\mu_{j})=\gamma_{j}(\mu_{1},\dots,\mu_{j})$,\,\,\, for $1\leq j\leq e$.
\end{center}

The next step of our task is to define $\beta_{j}$ for $e<j<e^{'}$. Note that if $e^{'}=e+1$, there is nothing to do in this case. So assume that $e+d=e^{'}$ such that $d>1$. For any $1\leq j<d$, define $k_{j}$ (inductively) to be the least ordinal $<\theta$ such that $\gamma_{k_{j}}>\beta_{e+j-1}$ and also for the correspondent subset $X_{k_{j}}\subset F^{-1}(\gamma_{k_{j}})$, we have $|X_{k_{j}}|\geq \mu_{e+j}$. Thus for $1\leq j\leq d-1$ put
\begin{equation}\label{betagamma}
\beta_{e+j}(\mu_{1},\dots,\mu_{e+j})=\gamma_{k_{j}}(\mu_{1},\dots,\mu_{e}).
\end{equation}

The more challenging case is defining $\beta_{j}$'s for $e^{'}\!\leq j<\!\theta$, namely the last move of the player II, where the player I has played $\mu_{e^{'}}$ in his last move. Let $|M_{\beta_{\!e^{'}-1}}|=\pi_{*}$ and for simplicity denote $M_{\beta_{\!e^{'}\!\!-1}}$ by $M_{*}$. Let $\langle\pi_{i};i<\theta\rangle$ be a sequence of strictly increasing cardinals $<\theta$ such that $\pi_{0}\geq\max\{2^{\pi_{*}}\!,\mu_{e^{'}}\}$. By induction we define a strictly increasing function
\[
g\colon \theta\longrightarrow \{i\,;\,\,k_{d-1}+1\leq i<\theta\}
\]
such that $g(i)$ is the least ordinal such that $|X_{g(i)}|\geq(\beth_{r-1}(\pi_{i}))^{+}$. In fact the strong inaccessibility of $\theta$ and the relation (\ref{theta}) guarantee the existence of such $g$. Note that if $e+1=e^{'}$, we replace $k_{d-1}$ by $e$ in the definition of $g$.
In continuation we need to find some suitable subsets $Z_{g(i)}$ of $X_{g(i)}$ for $i<\theta$ by using the Erd\"{o}s-Rado partition theorem \ref{erdosradop}. For any $i<\theta$, any $\alpha\in M_{*}$ and any
\[
\textbf{a}\in \displaystyle{\prod_{i=1}^{e}\big{[}X_{i}\big{]}^{r}}\times\displaystyle{\prod_{i=1}^{d-1}\big{[}X_{k_{i}}\big{]}^{r}},
\]
put
\[
P^{i}_{\textbf{a},\alpha}=\bigr\{\mathbf{x}\in[X_{g(i)}]^{r}; \tau(\textbf{a},\mathbf{x})=\alpha\bigr\},
\]
where $\tau$ is as mentioned in the first line of the proof (note that $\langle\mathbf{a},\mathbf{x}\rangle\in[F]^{r,f^{'}}$ and according to our convention, $\tau(\mathbf{a},\mathbf{x})$ is well-defined). Also suppose $\star$ is a new symbol different from all elements of $M_{*}$. For the above mentioned $i<\theta$ and $\textbf{a}$ put also
\[
P^{i}_{\textbf{a},\star}=\bigr\{\mathbf{x}\in[X_{g(i)}]^{r}; \tau(\textbf{a},\mathbf{x})>M_{*}\bigr\}.
\]
It is evident that fixing $i$ and $\textbf{a}$ as above, the set $\bigr\{P^{i}_{\textbf{a},\alpha}|\alpha\in M_{*}\cup\{\star\}\bigr\}$ becomes a partition of $[X_{g(i)}]^{r}$. We denote the partition relation by $\mathcal{R}^{i}_{\textbf{a}}$. In other words for any
$\mathbf{x}_1,\mathbf{x}_2$ in $[X_{g(i)}]^{r}$,
we have
$\mathbf{x}_1\mathcal{R}^{i}_{\textbf{a}}\mathbf{x}_2$
iff there exists $\alpha\in M_{*}\cup\{\star\}$ such that
$\mathbf{x}_1,\mathbf{x}_2\in P^{i}_{\textbf{a},\alpha}$.
Now for any $i<\theta$, let $\mathcal{R}^{i}$ be the following partition relation:
\begin{center}
$\forall\,\textbf{x}_{1},\textbf{x}_{2}\in [X_{g(i)}]^{r}$:\,\,\, $\textbf{x}_{1}\,\mathcal{R}^{i}\,\textbf{x}_{2}$ \,\, iff \,\, $\forall \,\textbf{a}\in\displaystyle{\prod_{i=1}^{e}\big{[}X_{i}\big{]}^{r}}\times\displaystyle{\prod_{i=1}^{d-1}\big{[}X_{k_{i}}\big{]}^{r}}\!\!,\,\,
\textbf{x}_{1}\,\mathcal{R}^{i}_{\textbf{a}}\,\textbf{x}_{2}$.
\end{center}
All $\mathcal{R}^{i}$'s have the same number of partition classes, that is, it does not depend on $i<\theta$. Let $\chi$ be the cardinality of the partition classes, then it is easily seen that
\[
\chi\leq|M_{*}|^{|M_{*}|}=2^{|M_{*}|}=2^{\pi_{*}}\leq\pi_{0}.
\]
Note that in a partition relation we can make the cardinals in the right side of the relation, smaller and also the cardinals in the left side of the relation, bigger. So by the Erd\"{o}s-Rado partition relation, for any $i<\theta$, we have
\begin{center}
$\beth_{r-1}(\pi_{i})^{+}\longrightarrow(\pi_{i}^{+})^{r}_{\chi}$.
\end{center}
Recall $|X_{g(i)}|\geq(\beth_{r-1}(\pi_{i}))^{+}$, therefor for $i<\theta$ there is a subset $Z_{g(i)}\subset X_{g(i)}$ such that $[Z_{g(i)}]^{r}$ lies in one partition class of $\mathcal{R}^{i}$ and $|Z_{g(i)}|=\pi_{i}^{+}$. This means that for each $i<\theta$ there is a function
\[
G_{i}\colon\displaystyle{\prod_{i=1}^{e}\big{[}X_{i}\big{]}^{r}}\times\displaystyle{\prod_{i=1}^{d-1}\big{[}X_{k_{i}}\big{]}^{r}}\longrightarrow M_{*}\cup\{\star\}
\]
such that if $G_{i}(\textbf{a})=\alpha\in M_{*}$, then for all $\mathbf{x}\in[Z_{g(i)}]^{r}$ we have
$\tau(\textbf{a},\mathbf{x})=\alpha$
and if $G_{i}(\textbf{a})=\star$, then for all $\mathbf{x}\in[Z_{g(i)}]^{r}$ we have
$\tau(\textbf{a},\mathbf{x})>M_{*}$.

Since the cardinality of all such functions is at most $|M_{*}|^{|M_{*}|}<\theta$, then there is a strictly increasing function $h\colon\theta\longrightarrow\theta$, such that for any $i,j<\theta$ we have
\begin{equation}\label{G}
G_{h(i)}=G_{h(j)}.
\end{equation}

Now we are ready to define the desired $\langle \beta_{e^{'}+i};1\leq i<\theta\rangle$ as follows:
\begin{equation}\label{lastbeta}
\beta_{e^{'}+i}=\gamma_{g(h(i))},\,\, i<\theta.
\end{equation}

After completing the description of the strategy of the player II in the game $G(S^{'},e^{'})$, it remains to show that it is a winning strategy. Clearly our definitions implies that $\beta_{i}$'s are strictly increasing. Then we must show that there are subsets
\begin{equation}\label{Y}
Y_{1}\in[F^{-1}(\beta_{1})]^{\mu_{1}},\dots,Y_{e^{'}}\in[F^{-1}
(\beta_{e^{'}})]^{\mu_{\!e^{'}}}
\end{equation}
together with subsets
\begin{equation}\label{subset2}
Y_{e^{'}+i}\subset F^{-1}(\beta_{i})
\end{equation}
for $i<\theta$ such that
\begin{equation}\label{theta2}
\sup\bigr\{|Y_{e^{'}+i}|;i<\theta\bigr\}=\theta
\end{equation}
and
\begin{equation}\label{big2}
\displaystyle{\prod_{1\leq i\leq e^{'}}}\big{[}Y_{i}\big{]}^{r}
\times \big{[}F|(\displaystyle{\bigcup_{1\leq i<\theta}}Y_{e^{'}+i})\big{]}^{r,f^{'}}\subset S^{'},
\end{equation}
where $f^{'}=s-e^{'}$.

Our strategy to define $Y_{i}$ will be as follows: we first define sets $Y_{i}^{*}$ such that they satisfy the relations (\ref{Y}), (\ref{subset2}), (\ref{theta2}). Then by the support of Lemma \ref{combinatoriallemma} we will find $Y_{i}\in (Y_{i}^{*})^{\bullet\bullet}$ which satisfy (\ref{big2}). Obviously $Y_{i}$ will automatically satisfy (\ref{Y}), (\ref{subset2}), (\ref{theta2}).

For $1\leq i\leq e$, let $Y_{i}^{*}=X_{i}$ and for $e<i<e^{'}$, let $Y_{i}^{*}=X_{k_{i-e}}$. Also for $i<\theta$, let $Y_{e^{'}+i}^{*}=Z_{g(h(i))}$.
The corresponding relations (\ref{subset2}), (\ref{Y}) hold for $Y_{i}^{*}$ because
\begin{center}
$Y_{i}^{*}=X_{i}\in [F^{-1}(\gamma_{i})]^{\mu_{i}}=[F^{-1}(\beta_{i})]^{\mu_{i}}$,\,\,\, for $1\leq i\leq e$.
\end{center}
\begin{center}
$Y_{i}^{*}=X_{k_{i-e}}\in [F^{-1}(\gamma_{k_{i-e}})]^{\mu_{i}}=[F^{-1}(\beta_{i})]^{\mu_{i}}$,\,\,\, for $e<i<e^{'}$.
\end{center}
\begin{center}
$Y_{e^{'}+i}^{*}= Z_{g(h(i))}\subset X_{g(h(i))}\subset F^{-1}(\gamma_{g(h(i))})=F^{-1}(\beta_{e^{'}+i})$,\,\,\, for $i<\theta$.
\end{center}
Note that since $h\colon\theta\longrightarrow\theta$ is a strictly increasing function, then we have $h(i)\geq i$ for each $i<\theta$, hence for $i<\theta$:
\[
|Y_{e^{'}+i}^{*}|=|Z_{g(h(i))}|\geq |Z_{g(i)}|\geq \pi_{i}^{+},
\]
So $\sup\bigr\{|Y_{e^{'}+i}^{*}|;i<\theta\bigr\}=\theta$. Also $|Y_{e^{'}}^{*}|\geq\pi_{0}^{+}>\pi_{0}\geq\mu_{e^{'}}$. Of course this will not cause a problem since we can easily replace $Y_{e^{'}}^{*}$ by each one of its subsets of cardinality $\mu_{e^{'}}$.  Also it is not hard to see that
\begin{equation}\label{Y*}
\displaystyle{\prod_{1\leq i\leq e^{'}}}\big{[}Y_{i}^{*}\big{]}^{r}
\times\big{[}F|(\displaystyle{\bigcup_{1\leq i<\theta}}Y_{e^{'}+i}^{*})\big{]}^{r,f^{'}}\subset S.
\end{equation}
[Why? Observe that
\begin{center}
$\displaystyle{\prod_{1\leq i\leq e^{'}}}\big{[}Y_{i}^{*}\big{]}^{r}
\times \big{[}F|(\displaystyle{\bigcup_{1\leq i<\theta}}Y_{e^{'}+i}^{*})\big{]}^{r,f^{'}}
=
\displaystyle{\prod_{1\leq i\leq e}}\big{[}Y_{i}^{*}\big{]}^{r}
\times
\displaystyle{\prod_{e<i\leq e^{'}}}\big{[}Y_{i}^{*}\big{]}^{r}
\times
\big{[}F|(\displaystyle{\bigcup_{1\leq i<\theta}}Y_{e^{'}+i}^{*})\big{]}^{r,f^{'}}$.
\end{center}
The right side of the above equality can be rewritten as
\begin{center}
$\displaystyle{\prod_{1\leq i\leq e}}\big{[}X_{i}\big{]}^{r}
\times
\displaystyle{\prod_{e<i<e^{'}}}\big{[}X_{k_{i-e}}\big{]}^{r}
\times
\big{[}Z_{g(h(0))}\big{]}^{r}
\times
\big{[}F|(\displaystyle{\bigcup_{1\leq i<\theta}}Z_{g(h(i))})\big{]}^{r,f^{'}}$
\end{center}
which is a subset of
\begin{center}
$(\clubsuit)\displaystyle{\prod_{1\leq i\leq e}}\big{[}X_{i}\big{]}^{r}
\times
\displaystyle{\prod_{e<i<e^{'}}}\big{[}X_{k_{i-e}}\big{]}^{r}
\times
\big{[}X_{g(h(0))}\big{]}^{r}
\times
\big{[}F|(\displaystyle{\bigcup_{1\leq i<\theta}}X_{g(h(i))})\big{]}^{r,f^{'}}$
\end{center}
But we have $(d-1)+1+f^{'}=f$ and for $i<\theta$
\begin{center}
$e<k_{1}<\dots<k_{d-1}<g(h(0))<g(h(1))<\dots<g(h(i))<\cdots$
\end{center}
so we deduce that ($\clubsuit$) is contained in
\begin{center}
$\displaystyle{\prod_{1\leq i\leq e}}\big{[}X_{i}\big{]}^{r}
\times \,\,\,\big{[}F|(\!\!\displaystyle{\bigcup_{1\leq i<\theta}}X_{e+i})\big{]}^{r,f}$
\end{center}
which is a subset of $S$ by (\ref{big}). Thus we have proved (\ref{Y*}).]

Now for the moment we digress from the sentence $\sigma$ and consider a related sentence $\sigma^{*}$. Let $\sigma^{*}$ be the sentence obtained from $\sigma$ as follows: we replace indices $l_1,\dots,\l_{n-q}$ by $j_{q+1},\dots,j_{n}$ respectively. We claim that
\begin{equation}\label{sigma*}
\forall\,\textbf{x}\in\displaystyle{\prod_{1\leq i\leq e^{'}}}\big{[}Y_{i}^{*}\big{]}^{r}
\times\big{[}F|(\displaystyle{\bigcup_{1\leq i<\theta}}Y_{e^{'}+i}^{*})\big{]}^{r,f^{'}} M\models\sigma^{*}(\textbf{x}).
\end{equation}
Suppose
\[
\textbf{g}=\langle \textbf{g}_{1},\dots,\textbf{g}_{s}\rangle\in\displaystyle{\prod_{1\leq i\leq e^{'}}}\big{[}Y_{i}^{*}\big{]}^{r}
\times\big{[}F|(\displaystyle{\bigcup_{1\leq i<\theta}}Y_{e^{'}+i}^{*})\big{]}^{r,f^{'}}
\]
and for $1\leq i\leq s$, $\textbf{g}_{i}=\langle g_{i1},\dots,g_{ir}\rangle$, so if $\tau(g_{i_{1}j_{1}},\ldots,g_{e^{'}j_{n}})\geq g_{(e^{'}-1)j}$,
then obviously $M\models\sigma^{*}(\textbf{g})$. So we assume that
\begin{equation}\label{b1}
\tau(g_{i_{1}j_{1}},\ldots,g_{e^{'}j_{n}})<g_{(e^{'}-1)j},
\end{equation}
but $g_{(e^{'}-1)j}\in Y^{*}_{e^{'}-1}\subset F^{-1}(\beta_{e^{'}-1})\subset M_{*}$, therefore we must show
\begin{equation}\label{b2}
\tau(\underline{g},g_{e^{'}j_{q+1}},\ldots,g_{e^{'}j_{n}})=\tau(\underline{g},g_{uj_{q+1}},\ldots,g_{uj_{n}}),
\end{equation}
where $\underline{g}=\langle g_{i_{1}j_{1}},\ldots,g_{i_{q}j_{q}}\rangle$, $u>e^{'}$ and $q$ is the greatest integer such that $i_{q}\neq e^{'}.$ For $1\leq i\leq e^{'}$ we have $\textbf{g}_{i}\in[Y_{i}^{*}]^{r}$. Let $v_{1}<\dots<v_{f^{'}}<\theta$ be such that
\begin{center}
$\textbf{g}_{e^{'}}\in[Y_{e^{'}}]^{r}, \textbf{g}_{e^{'}+1}\in[Y_{e^{'}+v_{1}}]^{r},\dots,\textbf{g}_{e^{'}+f^{'}}\in[Y_{e^{'}+v_{f^{'}}}]^{r}$ .
\end{center} Also assume that $\langle \textbf{g}_{1},\dots,\textbf{g}_{e^{'}-1}\rangle=\textbf{a}$. In order to avoid ambiguity when replacing $c_{ij}$'s by $\textbf{g}$ in term $\tau$, we define
\begin{eqnarray*}
\tau^{\mathrm{right}}&=&\tau(c_{i_{1}j_{1}},\ldots,c_{i_{q}j_{q}},c_{uj_{q+1}},\ldots,c_{uj_{n}}),\\
\tau^{\mathrm{left}}\,\,\,&=&\tau(c_{i_{1}j_{1}},\ldots,c_{i_{q}j_{q}},c_{i_{q+1}j_{q+1}},\ldots,c_{i_{n}j_{n}}).
\end{eqnarray*}
Hence the equation (\ref{b2}) equivalently can be rewritten as
\begin{equation}\label{b3}
\tau^{\mathrm{left}}(\textbf{a},\textbf{g}_{e^{'}})=\tau^{\mathrm{right}}(\textbf{a},\textbf{g}_{e^{'}+u}),\,\,\,1\leq u\leq f^{'}.
\end{equation}
Recall that
\[
\textbf{a}\in \displaystyle{\prod_{i=1}^{e}\big{[}X_{i}\big{]}^{r}}\times\displaystyle{\prod_{i=1}^{d-1}\big{[}X_{k_{i}}\big{]}^{r}}
=\displaystyle{\prod_{i=1}^{e^{'}}}\big{[}\,Y_{i}^{*}\big{]}^{r},
\]
$Y_{e^{'}}^{*}=Z_{g(h(0))}$ and for $1\leq j\leq f^{'}$, $Y_{e^{'}+v_{j}^{*}}=Z_{g(h(v_{j}))}$. By (\ref{G}) we have
\[
G_{h(0)}(\,\textbf{a})=G_{h(v_{j})}(\,\textbf{a})\,\,\in M_{*}\cup\{\star\},
\]
which means that \underline{either}, there is an $\alpha\in M_{*}$ such that for all $\textbf{z}\in [Y_{e^{'}}^{*}]^{r}=[Z_{g(h(0))}]^{r}$ and all $\textbf{z}^{'}\in[Y_{e^{'}+v_{j}}^{*}]^{r}=[Z_{g(h(v_{j}))}]^{r}$ we have
\begin{equation}\label{equalalpha}
\tau^{\mathrm{left}}(\textbf{a},\textbf{z})=\tau^{\mathrm{left}}(\textbf{a},\textbf{z}^{'})=\alpha
\end{equation}
\underline{or}, for all $\textbf{z}\in [Y_{e^{'}}^{*}]^{r}=[Z_{g(h(0))}]^{r}$ we have
\begin{equation}\label{greaterthanM}
\tau^{\mathrm{left}}(\textbf{a},\textbf{z})>M_{*}.
\end{equation}
According to (\ref{b1}), we deduce that the relation (\ref{greaterthanM}) cannot happen, so by (\ref{equalalpha}) for all $1\leq u\leq f^{'}$ we have
\[
\tau^{\mathrm{left}}(\textbf{a},\textbf{g}_{e^{'}})=\tau^{\mathrm{left}}(\textbf{a},\textbf{g}_{e^{'}+u}).
\]
Since $1\leq u$ the relation (\ref{equalalpha}) also implies that
\[
\tau^{\mathrm{left}}(\textbf{a},\textbf{g}_{e^{'}+u})=\tau^{\mathrm{right}}(\textbf{a},\textbf{g}_{e^{'}+u}),
\]
which implies that
\[
\tau^{\mathrm{left}}(\textbf{a},\textbf{g}_{e^{'}})=\tau^{\mathrm{right}}(\textbf{a},\textbf{g}_{e^{'}+u}).
\]
This proves what we claimed in (\ref{sigma*}).

Now for $0<i<\theta$ let $Y_{i}$ be any member of $(Y_{i}^{*})^{\bullet\bullet}$. By (\ref{Y*}) we have
\begin{equation}\label{YY}
\displaystyle{\prod_{1\leq i\leq e^{'}}}\big{[}Y_{i}\big{]}^{r}
\times\big{[}F|(\displaystyle{\bigcup_{1\leq i<\theta}}Y_{e^{'}+i})\big{]}^{r,f^{'}}\subset S.
\end{equation}
Note that the the following two sequences are equivalent:
\[
\langle c_{i_{1}j_{1}},\dots,c_{i_{q}j_{q}},c_{e^{'}j_{q+1}},\dots,c_{e^{'}j_{n}},c_{uj_{q+1}},\dots,c_{uj_{n}}\rangle
\]
\[
\langle c_{i_{1}j_{1}},\dots,c_{i_{q}j_{q}},c_{e^{'}j_{q+1}},\dots,c_{e^{'}j_{n}},c_{ul_{1}},\dots,c_{ul_{n-q}}\rangle
\]

The first sequence is the set of all constant symbols appearing in $\sigma^{*}$ and the second sequence shows the set of all constant symbols appearing in $\sigma$. Now from the claim (\ref{sigma*}) and Lemma \ref{combinatoriallemma}, it follows that
\begin{equation}\label{sigma}
\forall\,\mathbf{x}\in\displaystyle{\prod_{1\leq i\leq e^{'}}}\big{[}Y_{i}\big{]}^{r}
\times\big{[}F|(\displaystyle{\bigcup_{1\leq i<\theta}}Y_{e^{'}+i})\big{]}^{r,f^{'}} M\models\sigma(\mathbf{x}).
\end{equation}
Putting together the relations (\ref{sigma}), (\ref{YY}) and also the definition of $S^{'}$, we deduce that
\[
\displaystyle{\prod_{1\leq i\leq e^{'}}}\big{[}Y_{i}\big{]}^{r}
\times\big{[}F|(\displaystyle{\bigcup_{1\leq i<\theta}}Y_{e^{'}+i})\big{]}^{r,f^{'}}\subset S^{'}
\]
which is exactly what we wanted in (\ref{big2}). This finishes the proof of Proposition \ref{proposition1}.
\end{proof}

\begin{proposition}\label{proposition2}
Let $S\subset [F]^{r,s}$ be an e-superlarge set $(e<s)$. Suppose $\sigma_{1},\dots,\sigma_{p}$ are any finitely many sentences of type $\Sigma_{1}^{*}(iv)$ so that  for all $c_{ij}$ occurring in $\sigma$ we have $i\leq s$ and $j\leq r$. Let $\iota(\sigma_{1})=\dots=\iota(\sigma_{p})=e^{'}>e$. Then there is an $e^{'}$-superlarge set $S^{'}\subset S$ such that for any $\mathbf{a}\in S^{'}$, $M\models\sigma_{1}(\mathbf{a})\wedge\dots\wedge\sigma_{p}(\textbf{a})$.
\end{proposition}
\begin{proof} The proof is almost the same as the proof of Proposition \ref{proposition1}. The only difference is that this time we must take into account all of $\sigma_1,\dots,\sigma_p$ simultaneously when we use the Erd\"{o}s-Rado partition theorem which can be done with no more difficulty, so we leave it to the reader.
\end{proof}

\begin{theorem}\label{partial consistency}
 $T+\Sigma_{1}$ is consistent.
\end{theorem}
\begin{proof}
It is enough to show that $T+\Sigma_{1}^{*}$ is consistent. Let $\Sigma_{1}^{'}$ be a finite part of $\Sigma_{1}^{*}$. Suppose $r,s$ are large enough positive integers such that for any $\sigma\in\Sigma_{1}^{'}$ and any $c_{ij}$ occurring in $\sigma$ we have $i\leq s$ and $j\leq r$. We also interpret naturally all symbols of $\mathcal{L}^{S}$ in $M$. So $M\models T_{\mathrm{Skolem}}$. Our aim is to find an $\textbf{a}\in[F]^{r,s}$ such that for each $\sigma\in\Sigma_{1}^{'}$, we have $M\models\sigma(\textbf{a})$. Therefore the compactness theorem will imply that $T+\Sigma_{1}^{*}$ is consistent. First suppose that $\sigma\in\Sigma_{1}^{'}$ is a sentence of type $\Sigma_{1}^{*}$(ii), by definition it is clear that for any $\textbf{a}\in[F]^{r,s}$ we have $a_{ij}<a_{kl}$ iff $(i,j)<(k,l)$ lexicographically, where $1\leq i,k\leq s$ and $1\leq j,l\leq r$. So for this type of $\sigma$, $M\models\sigma(\textbf{a})$. Now let $\sigma\in\Sigma_{1}^{'}$ is sentence of type $\Sigma_{1}^{*}$(iii). Consider any
\[
\textbf{a}=\langle \textbf{a}_{1},\dots,\textbf{a}_{s}\rangle=\big{\langle}\langle a_{11},\dots,a_{1r}\rangle,\ldots,
\langle a_{s1},\ldots,a_{sr}\rangle\big{\rangle}\in[F]^{r,s}
\]
and let $\tau(x_{1},\dots,x_{m})$ be the term appearing in $\sigma$. Recall that we had constructed $F\colon M\longrightarrow\theta$ in such a way that for any $\{a_{1},\dots,a_{m},b\}\subset M$:
\begin{center}
if $F(b)>\max(F(a_{1}),\dots,F(a_{m}))$, then $\tau(a_{1},\dots,a_{m})<b$.
\end{center}
This implies that $F(\textbf{a}_{1},\dots,\textbf{a}_{s})<a_{s1}$, since by the definition of $[F]^{r,s}$ we must have
\[
F(a_{s1})>F(a_{(s-1)r})=\dots=F(a_{(s-1)1})>\dots>F(a_{1r})=\dots=F(a_{11}).
\]
Finally assume that $B=\{\sigma_{1},\dots,\sigma_{p}\}$ is the set of all sentences of type $\Sigma_1$(iv) that has occurred in $\Sigma_{1}^{'}$. Set $A=\{\iota(\sigma_{1}),\dots,\iota(\sigma_{p})\}=\{e_{1},\dots,e_{q}\}$ such that $e_{1}<\dots<e_{q}$. Obviously $1<e_{1}$ and $e_{q}\leq s$ and $[F]^{r,s}$ is 1-superlarge. By a successive use of Proposition \ref{proposition2}, $q$ times, we can find subsets $S_{q}\subset\dots\subset S_{1}\subset[F]^{r,s}$ such that for $1\leq i\leq q$, every $S_{i}$ is $e_{i}$-superlarge and if $\textbf{a}\in S_{i}$, then $M\models\sigma(\textbf{a})$, where $\sigma\in B$ and $\iota(\sigma)=e_{i}$. Putting together all these, we have shown that for all $\textbf{a}\in S_{q}$ and all $\sigma\in\Sigma_{1}^{'}$ we have $M\models\sigma(\textbf{a})$. This completes the proof.
\end{proof}

\section{Second combinatorial proposition and proof of the second part of Theorem \ref{maintheorem}}

Keisler in \cite{keislerordering} introduced the following $\mathcal{L}^{S}(C)$-theory $\Sigma$:
\begin{definition}
Items $(i)$, $(ii)$ and $(iii)$ of $\Sigma$ are exactly the items $(i)$, $(ii)$ and $(iii)$ of $\Sigma_{1}$ and
\begin{itemize}
\item[(iv)] If $\tau(c_{i_{1}j_{1}},\dots,c_{i_{n}j_{n}})<c_{uv}$ where $\tau$ is a term of $\mathcal{L}^{S}$ and $u<i_{n}$ then
\begin{center}
$\tau(\overline{c},c_{i_{m+1}j_{m+1}},\dots,c_{i_{n}j_{n}})=\tau(\overline{c},c_{i_{m+1}l_{m+1}},\dots,c_{i_{n}l_{n}})$,
\end{center}
where $\overline{c}=\langle c_{i_{1}j_{1}},\dots,c_{i_{m}j_{m}}\rangle$ in which $m$ is the smallest integer such that $i_{m+1}>u$ and $u, v,l_{m+1}$$,\dots,l_{n}$ are arbitrary. If there is no such $m$, then the above equation becomes:
\begin{center}
$\tau(c_{i_{1}j_{1}},\dots,c_{i_{n}j_{n}})=\tau(c_{i_{1}l_{1}},\dots,c_{i_{n}l_{n}})$.
\end{center}
\end{itemize}
\end{definition}

We need to prove another combinatorial property of the superlarge sets, but we first note that $\Sigma$ is homogenous too. Now let $\Sigma^{*}$ be the $\mathcal{L}^{S}(C^{*})$-theory such that its sentences are exactly the sentences of $\Sigma$ except that this time the constants $c_{ij}$'s come from the set $C^{*}$. Again by homogeneouity, it is easy to see that

\begin{lemma}\label{sigmainfinite} For any $\mathcal{L}^{S}$-theory $\Gamma$, $\Gamma+\Sigma+\Sigma_{1}$ is consistent iff $\Gamma+\Sigma^{*}+\Sigma_{1}^{*}$ is consistent.
\end{lemma}

Suppose $\sigma$ is a sentence of type $\Sigma^{*}$(iv), we extend the domain of the function $\iota$ to such $\sigma$ and define $\iota(\sigma)=i_{n}$.

\begin{proposition}\label{proposition3}
Let $S\subset[F]^{rs}$ be an $e$-superlarge set $(e\leq s)$. Suppose $\sigma$ is a sentence of type $\Sigma^{*}(iv)$ so that for all $c_{ij}$ occurring in $\sigma$ we have $i\leq s$ and $j\leq r$. Let $\iota(\sigma)=e^{'}\geq e$, then there is an $e^{'}$-superlarge set $S^{'}\subset S$ such that for any $\mathbf{a}\in S^{'}$, $M\models \sigma(\mathbf{a})$.
\end{proposition}
\begin{proof}
First suppose that $\tau^{\mathrm{left}}$ and $\tau^{\mathrm{right}}$ are the terms occurring in the left and the right sides of the conclusion part of the sentence $\sigma$, respectively.  More precisely:
\begin{center}
$\tau^{\mathrm{left}}=\tau(\overline{c},c_{i_{m+1}j_{m+1}},\dots,c_{i_{n}j_{n}})$, \,\,\,
$\tau^{\mathrm{right}}=\tau(\overline{c},c_{i_{m+1}l_{m+1}},\dots,c_{i_{n}l_{n}})$
\end{center}
with $\overline{c}=\langle c_{i_{1}j_{1}},\dots,c_{i_{m}j_{m}}\rangle$.
 Assume that
\[
S^{'}=\{\textbf{a}\in S \, | \, M\models\sigma(\textbf{a})\}.
\]
We will show that $S^{'}$ is $e^{'}$-superlarge. This will be done if we can show that there is a winning strategy
\[
\beta_{1}(\mu_{1}),\dots,\beta_{e^{'}}(\mu_{1},\dots,\mu_{e^{'}}),\dots,\beta_{e^{'}+i}(\mu_{1},\dots,\mu_{e^{'}}),\dots\,\,\,i<\theta
\]
for the player II in the game $G(S^{'},e^{'})$. Suppose the player I plays according to the following strategy:
\[
\mu_{1},\mu_{2},\dots,\mu_{e^{'}}.
\]
Since $S$ is $e$-superlarge, then the player II has a winning strategy:
\[
\gamma_{1},\dots,\gamma_{e},\dots,\gamma_{e+i},\dots\,\,\, i<\theta
\]
for the game $G(S,e)$. Put $i_{m+1}-1=p$ (if there is no $m$ such that $i_{m+1}>u$, then put $p=i_{1}-1$ and note that $i_{1}>u\geq 1$). There are several cases to be considered. Case I: $e\geq p$. Case II: $e<p$.\\

\underline{Case I}: ($e\geq p$)\\

First recall the definition of the elementary end extension chain of initial submodels $\langle M_{i};i<\theta\rangle$ from the previous section. For simplicity we denote $M_{\gamma_{p}}$ by $M_{*}$ and set $|M_{*}|=\chi$. Assume that $\star$ is a new symbol different from any element of $M$. In this case we face with three subcases: Subcase (Ia): $e=p$. Subcase (Ib): $p<e=e^{'}$. Subcase (Ic): $p<e<e^{'}$.\\

\underline{\underline{Subcase (Ia)}}: ($e=p$)\\

Let $e^{'}-p=d$ where $d>0$. Suppose the following are the ordinals given by the wining strategy of the player II against the above mentioned strategy of player I in the game $G(S,e)$:
\[
\gamma_{1}(\mu_{1}),\dots,\gamma_{e}(\mu_{1},\dots,\mu_{e}),\dots,\gamma_{e+i}(\mu_{1},\dots,\mu_{e}),\dots\,\,\, i<\theta
\]
This implies that $\gamma_{1}<\gamma_{2}<\dots<\gamma_{i}<\dots$ for $i<\theta$ and there exist sets:
\[
X_{1}\in[F^{-1}(\gamma_{1})]^{\mu_{1}},\dots,X_{e}\in[F^{-1}(\gamma_{e})]^{\mu_{e}}
\]
as well as the following sets:
\[
X_{e+i}\subset F^{-1}(\gamma_{e+i})\,\,\,\mathrm{for}\,\,\,1\leq i<\theta
\]
such that
\[
\sup\bigr\{|X_{e+i}|;i<\theta\bigr\}=\theta
\]
and
\[
\displaystyle{\prod_{i=1}^{e}}\big{[}X_{i}\big{]}^{r}\times\big{[}F|(\bigcup_{1\leq i<\theta}X_{e+i})\big{]}^{r,f}\subset S,
\]
where $f=s-e$. Now we move towards defining $\beta$'s which guarantee the winning of the player II in the game $G(S^{'},e^{'})$.
Let
\[
\beta_{j}(\mu_{1},\dots,\mu_{j})=\gamma_{j}(\mu_{1},\dots,\mu_{j})\,\,\,\mathrm{for}\,\,\,1\leq j\leq p.
\]
Suppose $\mu_{p+1}$ is given. Put $\lambda_{1}=\max (\mu_{p+1},2^{\chi})$. Let $\kappa_{1}$ be a cardinal with $\theta>\kappa_{1}>\beth_{r-1}(\lambda_{1})$ and $\delta_{1}$ is the least ordinal such that $|X_{e+\delta_{1}}|\geq\kappa_{1}$. Now set
\[
\beta_{p+1}(\mu_{1},\dots,\mu_{p+1})=\gamma_{e+\delta_{1}}(\mu_{1},\dots,\mu_{e}).
\]
If $d=1$, then this completes the description of the strategy of the player II in the game $G(S^{'},e^{'})$. If $d>1$, then for $1< i\leq d$ suppose we have defined $\beta_{p+1},\dots,\beta_{p+(i-1)}$ and $\mu_{p+i}$ is given. Set $\lambda_{i}=\max (2^{\kappa_{i-1}},\mu_{p+i})$ and let $\kappa_{i}$ be any cardinal $>\beth_{r-1}(\lambda_{i})$ and $<\theta$. Suppose $\delta_{i}$ is the least ordinal $<\theta$ and $>\delta_{i-1}$ such that $|X_{e+\delta_{i}}|\geq\kappa_{i}$. Now we define
\[
\beta_{p+i}(\mu_{1},\dots,\mu_{p+i})=\gamma_{e+\delta_{i}}(\mu_{1},\dots,\mu_{e}).
\]
So far we have defined $\beta_{1},\dots,\beta_{e^{'}}$. For $1<i<\theta$ let
\[
\beta_{e^{'}+i}(\mu_{1},\dots,\mu_{e^{'}})=\gamma_{e+\delta_{d}+i}(\mu_{1},\dots,\mu_{e}).
\]
This completes the description of the strategy of the player II in the game $G(S^{'},e^{'})$. It remains to show that it is a winning strategy. We should find subsets $Y_{i}\in[F^{-1}(\beta_{i})]^{\mu_{i}}$ for $1\leq i\leq e^{'}$ as well as subsets $Y_{e^{'}+i}\subset F^{-1}(\beta_{e^{'}+i})$ for $i<\theta$ such that sup$\{|Y_{e^{'}+i}|; i<\theta\}=\theta$ and
\begin{equation}\label{hadaf1}
\displaystyle{\prod_{i=1}^{e^{'}}}\big{[}Y_{i}\big{]}^{r}\times\big{[}F|(\bigcup_{1\leq i<\theta}Y_{e^{'}+i})\big{]}^{r,f^{'}}\subset S^{'},
\end{equation}
where $s-e^{'}=f^{'}$. By Corollary \ref{usefulcorollary} of the polarized Erd\"{o}s-Rado partition theorem we have:
\begin{equation}\label{partitionia}
(\kappa_{1},\dots,\kappa_{d})\longrightarrow(\mu_{p+1}^{+},\dots,\mu_{e^{'}}^{+})^{r}_{2^{\chi}}.
\end{equation}
Now we shall introduce a partition relation $\mathcal{R}$ on the set
\[
[X_{e+\delta_{1}}]^{r}\times\dots\times[X_{e+\delta_{d}}]^{r}.
\]
Assume that $\star$ is a new symbol different from any element of $M$. Now for any $\alpha\in M_{*}\cup\{\star\}$ and
any $\mathbf{a}$ in $[X_{1}]^{r}\times\dots\times[X_{p}]^{r}$ let
\[
P_{\alpha,\mathbf{a}}=\bigr\{\mathbf{x}\in\big{[}X_{e+\delta_1}\big{]}^{r}\times\dots\times\big{[}X_{e+\delta_d}\big{]}^{r}
:\tau^{\mathrm{left}}(\mathbf{a},\mathbf{x})=\alpha\bigr\},
\]
where $\tau^{\mathrm{left}}(\mathbf{a},\mathbf{x})=\star$ is an abbreviation for $\tau^{\mathrm{left}}(\mathbf{a},\mathbf{x})>M_{*}$. It is evident that fixing $\textbf{a}$ as above, the set $\bigr\{P_{\textbf{a},\alpha}|\alpha\in M_{*}\cup\{\star\}\bigr\}$ becomes a partition of
\[
[X_{e+\delta_{1}}]^{r}\times\dots\times[X_{e+\delta_{d}}]^{r}.
\]
We denote the partition relation by $\mathcal{R}_{\mathbf{a}}$. Now we are ready to define $\mathcal{R}$:
\begin{center}
$\mathbf{x}_1\mathcal{R}\mathbf{x}_2$\,\,\,iff\,\,\, $\forall\,\mathbf{a}\in\displaystyle{\prod_{i=1}^{p}}[X_{i}]^{r}:\mathbf{x}_1\mathcal{R}_{\mathbf{a}}\mathbf{x}_2$
\end{center}
It is easy to see that the number of partition classes is at most $\chi^{\chi}=2^{\chi}$. Hence by (\ref{partitionia}), there are subsets $Z_{i}\subset X_{e+\delta_{i}}$ for $1\leq i\leq d$ such that $|Z_{i}|=\mu_{p+i}$ and the set
\[
[Z_{1}]^{r}\times\dots\times[Z_{d}]^{r}
\]
lies in one partition class. Now suppose for $1\leq i\leq p$: $Y_{i}^{*}=X_{i}$, for $1\leq i\leq d$: $Y_{e+i}^{*}=Z_{i}\,$ and for $1\leq i<\theta$: $Y_{e^{'}+i}=X_{e+\delta_{d}+i}$. Finally for $1\leq i\leq e^{'}$ let $Y_{i}$ be any member of $(Y_{i}^{*})^{\bullet\bullet}$ in the sense of Fact \ref{fact}. Now we can deduce that
\[
\forall\,\mathbf{a}\in\displaystyle{\prod_{i=1}^{p}}\big{[}Y_{i}^{*}\big{]}^{r}
\]
\underline{either}
\begin{equation}\label{yek}
\forall\,\mathbf{x}\in\displaystyle{\prod_{i=1}^{d}}\big{[}Y_{e+i}^{*}\big{]}^{r}\,\,\tau^{\mathrm{left}}(\mathbf{a},\mathbf{x})>M_{*},
\end{equation}
\underline{or} there exists $\alpha\in M_{*}$ such that
\begin{equation}\label{do}
\forall\,\mathbf{x}\in\displaystyle{\prod_{i=1}^{d}}\big{[}Y_{e+i}^{*}\big{]}^{r}\,\,\tau^{\mathrm{left}}(\mathbf{a},\mathbf{x})=\alpha.
\end{equation}
Now we move towards proving the required properties of $Y_i$. Of course for $1\leq i\leq e$:
\[
Y_{i}^{*}=X_{i}\in[F^{-1}(\gamma_{i})]^{\mu_{i}}=[F^{-1}(\beta_{i})]^{\mu_{i}},
\]
thus $Y_{i}\in[F^{-1}(\beta_{i})]^{\mu_{i}}$. Also for $1\leq i\leq d$ we have
\[
Y_{e+i}^{*}=Z_{i}\subset X_{e+\delta_{i}}\in[F^{-1}(\gamma_{e+\delta_{i}})]^{\kappa_{i}}
\]
and $|Z_{i}|=\mu_{e+i}$, hence $Y_{e+i}^{*}\in[F^{-1}(\beta_{e+i})]^{\mu_{p+i}}$ and
$Y_{e+i}\in[F^{-1}(\beta_{e+i})]^{\mu_{p+i}}$. For the rest we have:
\[
Y_{e^{'}+i}=X_{e+\delta_{d}+i}\subset F^{-1}(\gamma_{e+\delta_{d}+i})=F^{-1}(\beta_{e^{'}+i}),
\]
for $1\leq i <\theta$. Note also that
\begin{eqnarray*}
\theta &=& \sup\bigr\{|X_{e+i}|;i<\theta\bigr\}\\
       &=& \sup\bigr\{|X_{e+\delta_{d}+i}|;i<\theta\bigr\}\\
       &=& \sup\bigr\{|Y_{e^{'}+i}|;i<\theta\bigr\}.
\end{eqnarray*}
It remains to show that the inclusion (\ref{hadaf1}) holds. We first show that
\begin{equation}\label{hadaf2}
\displaystyle{\prod_{i=1}^{e^{'}}}\big{[}Y_{i}\big{]}^{r}\times\big{[}F|(\bigcup_{1\leq i<\theta}Y_{e^{'}+i})\big{]}^{r,f^{'}}\subset S.
\end{equation}
Obviously
\begin{equation}\label{bir}
\displaystyle{\prod_{i=1}^{e}}\big{[}Y_{i}\big{]}^{r}\times\displaystyle{\prod_{i=1}^{d}}\big{[}Y_{e+i}\big{]}^{r}=
\displaystyle{\prod_{i=1}^{e^{'}}}\big{[}Y_{i}\big{]}^{r},\,\,\,
\displaystyle{\prod_{i=1}^{e}}\big{[}Y_{i}\big{]}^{r}=\displaystyle{\prod_{i=1}^{e}}\big{[}X_{i}\big{]}^{r}
\end{equation}
as well as
\begin{equation}\label{iki}
\displaystyle{\prod_{i=1}^{d}}\big{[}Y_{e+i}\big{]}^{r}\times\big{[}F|(\bigcup_{1\leq i<\theta}Y_{e^{'}\!+i})\big{]}^{r,f^{'}}\subset
\big{[}F|(\bigcup_{1\leq i<\theta}Y_{e+i})\big{]}^{r,f^{'}\!\!+d}.
\end{equation}
Observe that $f^{'}+d=f^{'}+(e^{'}-e)=(f^{'}+e^{'})-e=s-e=f$. Since for every $1\leq i<\theta$ there is $1\leq j<\theta$ such that $Y_{e+i}\subset X_{e+j}$, then
\begin{equation}\label{uch}
\big{[}F|(\bigcup_{1\leq i<\theta}Y_{e+i})\big{]}^{r,f}\subset\big{[}F|(\bigcup_{1\leq i<\theta}X_{e+i})\big{]}^{r,f}.
\end{equation}
Therefore by (\ref{bir}),(\ref{iki}) and (\ref{uch}) we conclude that
\[
\displaystyle{\prod_{i=1}^{e^{'}}}\big{[}Y_{i}\big{]}^{r}\times\big{[}F|(\bigcup_{1\leq i<\theta}Y_{e^{'}+i})\big{]}^{r,f^{'}}\subset
\displaystyle{\prod_{i=1}^{e}}\big{[}X_{i}\big{]}^{r}\times\big{[}F|(\bigcup_{1\leq i<\theta}X_{e+i})\big{]}^{r,f}\subset S
\]
which proves (\ref{hadaf2}). In order to establish (\ref{hadaf1}) it suffices to show (recall the definition of $S^{'}$):
\begin{equation}\label{birbir}
\forall\,\mathbf{x}\in\displaystyle{\prod_{i=1}^{e^{'}}}\big{[}Y_{i}\big{]}^{r}\times\big{[}F|(\bigcup_{1\leq i<\theta}Y_{e^{'}+i})\big{]}^{r,f^{'}}\,\,\, M\models\sigma(\mathbf{x}).
\end{equation}
The maximum first index $i$ in the constants $c_{ij}$ occurring in $\sigma$ is $\iota(\sigma)=i_n=e^{'}$, thus it is enough to consider only that part of $\mathbf{x}$ which comes from $[Y_{1}]^{r}\times\dots\times[Y_{e^{'}}]^{r}$.
In other words it is enough to show
\begin{equation}\label{ikiiki}
\forall\,\mathbf{x}\in\displaystyle{\prod_{i=1}^{e^{'}}}\big{[}Y_{i}\big{]}^{r}\,\,\,M\models\sigma(\mathbf{x}).
\end{equation}
Let $\mathbf{h}$ be an element of $\langle\mathbf{h}_1,\dots,\mathbf{h}_{e^{'}}\rangle\in[Y_1]^{r}\times\dots\times[Y_{e^{'}}]^{r}$. Let $\mathbf{a}=\langle\mathbf{h}_1,\dots,\mathbf{h}_p\rangle$, $\mathbf{b}=\langle\mathbf{h}_{p+1},\dots,\mathbf{h}_{e^{'}}\rangle$. Also for $1\leq i<e^{'}$, set $\mathbf{h}_i=\langle h_{i1},\dots,h_{ir}\rangle$. If $\tau(h_{i_1j_1},\dots,h_{i_nj_n})\leq h_{uv}$, then obviously $M\models\sigma(\mathbf{h})$. So suppose
$\tau(h_{i_1j_1},\dots,h_{i_nj_n})>h_{uv}$.
Then (\ref{ikiiki}) is reduced to
\begin{equation}\label{uchuch}
\tau^{\mathrm{left}}(\mathbf{h})=\tau^{\mathrm{right}}(\mathbf{h}).
\end{equation}
Recall that $u<i_{m+1}$, so $u\leq i_{m+1}-1=p$, then by $e=p$, we have $u\leq e$. This implies that $Y_{u}=X_{u}\subset F^{-1}(\gamma_{u})\subset M_{\gamma_p}=M_*$ and consequently $h_{uv}\in Y_u$ is a member of $M_*$. Since we have assumed that $\tau^{\mathrm{left}}(\mathbf{h})<h_{uv}$, it follows that $\tau^{\mathrm{left}}(\mathbf{h})\in M_*$. This will eliminate the possibility (\ref{yek}). Hence (\ref{do}) occurs. Thus there is an $\alpha\in M_{\star}$ such that
\begin{equation}\label{taazeh}
\forall\,\mathbf{y}\in\displaystyle{\prod_{i=1}^{d}}\big{[}Y_{e+i}^{*}\big{]}^{r}\,\,\tau^{\mathrm{left}}(\mathbf{a},\mathbf{y})=\alpha.
\end{equation}
Now suppose $\sigma_1,\sigma_2$ are the following two sentences:
\[
\sigma_1:\,\,\,\tau(\underline{h},c_{i_{m+1}j_{m+1}},\dots,c_{i_nj_n})=\alpha,
\]
\[
\sigma_2:\,\,\,\tau(\underline{h},c_{i_{m+1}l_{m+1}},\dots,c_{i_nl_n})=\alpha,
\]
where $\underline{h}=\langle h_{i_1j_1},\dots,h_{i_mj_m}\rangle$. From (\ref{taazeh}), it follows that
\begin{equation}\label{yekyek}
\forall\,\mathbf{y}\in\displaystyle{\prod_{i=p+1}^{e^{'}}}\big{[}Y_i^*\big{]}^r\,\,\,M\models\sigma_1(\mathbf{a},\mathbf{y}).
\end{equation}
But the two sequences $\langle c_{i_{m+1}j_{m+1}},\dots,c_{i_nj_n}\rangle$, $\langle c_{i_{m+1}l_{m+1}},\dots,c_{i_nl_n}\rangle$
are equivalent and hence Lemma \ref{combinatoriallemma} would imply
\begin{equation}\label{dodo}
\forall\,\mathbf{y}\in\displaystyle{\prod_{i=p+1}^{e^{'}}}\big{[}Y_i\big{]}^r\,\,\,M\models\sigma_2(\mathbf{a},\mathbf{y}).
\end{equation}
Putting (\ref{yekyek}) and (\ref{dodo}) together we obtain
\[
\forall\,\mathbf{y}\in\displaystyle{\prod_{i=p+1}^{e^{'}}}\big{[}Y_i\big{]}^r\,\,\,
\tau^{\mathrm{left}}(\mathbf{a},\mathbf{y})=\tau^{\mathrm{right}}(\mathbf{a},\mathbf{y}),
\]
which implies that $\tau^{\mathrm{left}}(\mathbf{a},\mathbf{b})=\tau^{\mathrm{right}}(\mathbf{a},\mathbf{b})$ and consequently $\tau^{\mathrm{left}}(\mathbf{h})=\tau^{\mathrm{right}}(\mathbf{h})$. This confirms (\ref{uchuch}) and finishes the proof of Subcase (Ia).\\

\underline{\underline{Subcase (Ib)}}: ($p<e=e^{'}$)\\

Let $e^{'}=e=p+d$, where $d>0$. We inductively define cardinals $\kappa_i,\lambda_i$ for $1\leq i\leq d$. If $d=1$, put $\lambda_1=\max (\mu_{p+1},2^{\chi})$ and $\kappa_1>\beth_{r-1}(\lambda_1)$. If $d>1$, then proceed as follows: for $2\leq i\leq d$ set $\lambda_i=\max (\kappa_{i-1},\mu_{p+i})$ and $\beth_{r-1}(\lambda_i)<\kappa_i<\theta$. Then by Corollary \ref{usefulcorollary} we have:
\begin{equation}\label{partitionib}
(\kappa_1,\dots,\kappa_d)\longrightarrow(\mu_{p+1},\dots,\mu_{e^{'}})^{r}_{2^{\chi}}.
\end{equation}
Now consider the following strategy of the player I in the game $G(S,e)$:
\[
\mu_1,\dots,\mu_p,\kappa_1,\dots,\kappa_d.
\]
Let the following be the ordinals given via the winning strategy of the player II for the game $G(S,e)$:
\begin{center}
$\gamma_1(\mu_1),\dots,\gamma_p(\mu_1,\dots,\mu_p)$,$\gamma_{p+1}(\mu_1,\dots,\mu_p,\kappa_1),\dots,\gamma_e(\mu_1,\dots,\mu_p,\kappa_1,\dots,\kappa_d)$,
\end{center}
\begin{center}

$\dots,\gamma_{e+i}(\mu_1,\dots,\mu_p,\kappa_1,\dots,\kappa_d),\dots$\,\,\, for $i<\theta$.
\end{center}
It follows that $\gamma_{1}<\gamma_{2}<\dots<\gamma_{i}<\dots$ for $i<\theta$ and there exist sets:
\begin{center}
$X_{1}\in[F^{-1}(\gamma_{1})]^{\mu_{1}},\dots,X_{p}\in[F^{-1}(\gamma_{p})]^{\mu_{p}}$,
\end{center}
\begin{center}
$X_{p+1}\in[F^{-1}(\gamma_{p+1})]^{\kappa_{1}},\dots,X_{e}\in[F^{-1}(\gamma_{e})]^{\kappa_{d}}$
\end{center}
as well as the sets:
\[
X_{e+i}\subset F^{-1}(\gamma_{e+i})\,\,\,\mathrm{for}\,\,\,1\leq i<\theta
\]
such that
\[
\sup\bigr\{|X_{e+i}|;i<\theta\bigr\}=\theta
\]
and
\begin{equation}\label{tak}
\displaystyle{\prod_{i=1}^{e}}\big{[}X_{i}\big{]}^{r}\times\big{[}F|(\bigcup_{1\leq i<\theta}X_{e+i})\big{]}^{r,f}\subset S,
\end{equation}
where $f=s-e$. Now we define $\beta_i$ which ensure that the player II wins the game $G(S^{'},e^{'})$.
Let
\begin{eqnarray*}
&&\beta_{i}(\mu_{1},\dots,\mu_{i}) = \gamma_{i}(\mu_{1},\dots,\mu_{i})\,\,\,\mathrm{for}\,\,\,1\leq i\leq p,\\
&&\beta_{p+i}(\mu_{1},\dots,\mu_{p+i}) = \gamma_{p+i}(\mu_{1},\dots,\mu_{p},\kappa_1,\dots,\kappa_i)\,\,\,\mathrm{for}\,\,\,1\leq i\leq d,\\
&&\beta_{e^{'}+i}(\mu_{1},\dots,\mu_{e^{'}}) = \gamma_{e+i}(\mu_{1},\dots,\mu_{p},\kappa_1,\dots,\kappa_d)\,\,\,\mathrm{for}\,\,\,1\leq i<\theta.
\end{eqnarray*}
Having completed the description of the strategy of the player II for the game $G(S^{'},e^{'})$, we shall show that it is a winning strategy. We would find subsets $Y_{i}\in[F^{-1}(\beta_{i})]^{\mu_{i}}$ for $1\leq i\leq e^{'}$ as well as subsets $Y_{e^{'}+i}\subset F^{-1}(\beta_{e^{'}+i})$ for $i<\theta$ such that sup$\{|Y_{e^{'}+i}|; i<\theta\}=\theta$ and
\begin{equation}\label{hadaf1ib}
\displaystyle{\prod_{i=1}^{e^{'}}}\big{[}Y_{i}\big{]}^{r}\times\big{[}F|(\bigcup_{1\leq i<\theta}Y_{e^{'}+i})\big{]}^{r,f^{'}}\subset S^{'},
\end{equation}
where $s-e^{'}=f^{'}$. Now we shall introduce a partition relation $\mathcal{R}$ on the set
\[
[X_{p+1}]^{r}\times\dots\times[X_{p+d}]^{r}.
\]
For any $\alpha\in M_{*}\cup\{\star\}$ and
any $\mathbf{a}$ in $[X_{1}]^{r}\times\dots\times[X_{p}]^{r}$ let
\[
P_{\alpha,\mathbf{a}}=\bigr\{\mathbf{x}\in\big{[}X_{p+1}\big{]}^{r}\times\dots\times\big{[}X_{p+d}\big{]}^{r}
:\tau^{\mathrm{left}}(\mathbf{a},\mathbf{x})=\alpha\bigr\},
\]
where $\tau^{\mathrm{left}}(\mathbf{a},\mathbf{x})=\star$ is an abbreviation for $\tau^{\mathrm{left}}(\mathbf{a},\mathbf{x})>M_{*}$. For every $\textbf{a}$ as above, the set $\bigr\{P_{\textbf{a},\alpha}|\alpha\in M_{*}\cup\{\star\}\bigr\}$ is a partition of
\[
[X_{p+1}]^{r}\times\dots\times[X_{p+d}]^{r}.
\]
We denote the produced partition relation by $\mathcal{R}_{\mathbf{a}}$. Let $\mathcal{R}$ be as follows:
\begin{center}
$\mathbf{x}_1\mathcal{R}\mathbf{x}_2$\,\,\,iff\,\,\, $\forall\,\mathbf{a}\in\displaystyle{\prod_{i=1}^{p}}[X_{i}]^{r}:\mathbf{x}_1\mathcal{R}_{\mathbf{a}}\mathbf{x}_2$
\end{center}
The number of partition classes is at most $2^{\chi}$. Hence by (\ref{partitionib}),there are subsets $Z_{i}\subset X_{p+i}$ for $1\leq i\leq d$ such that $|Z_{i}|=\mu_{p+i}$ and the set
\[
[Z_{1}]^{r}\times\dots\times[Z_{d}]^{r}
\]
lies in one partition class.

Now for $1\leq i\leq p$ put $Y_{i}^{*}=X_{i}$, for $1\leq i\leq d$ put $Y_{p+i}^{*}=Z_{i}\,$ and for $1\leq i<\theta$ set $Y_{e^{'}+i}=X_{e+i}$. Finally for $1\leq i\leq e^{'}$ let $Y_{i}$ be any member of $(Y_{i}^{*})^{\bullet\bullet}$ in the sense of Fact \ref{fact}. Now we can deduce that
\[
\forall\,\mathbf{a}\in\displaystyle{\prod_{i=1}^{p}}\big{[}Y_{i}^{*}\big{]}^{r}
\]
\underline{either}
\begin{equation}\label{yekib}
\forall\,\mathbf{x}\in\displaystyle{\prod_{i=1}^{d}}\big{[}Y_{p+i}^{*}\big{]}^{r}\,\,\tau^{\mathrm{left}}(\mathbf{a},\mathbf{x})>M_{*},
\end{equation}
\underline{or} there exists $\alpha\in M_{*}$ such that
\begin{equation}\label{doib}
\forall\,\mathbf{x}\in\displaystyle{\prod_{i=1}^{d}}\big{[}Y_{p+i}^{*}\big{]}^{r}\,\,\tau^{\mathrm{left}}(\mathbf{a},\mathbf{x})=\alpha.
\end{equation}
The next task is proving the required properties of $Y_i$. Of course for $1\leq i\leq p$:
\[
Y_{i}^{*}=X_{i}\in[F^{-1}(\gamma_{i})]^{\mu_{i}}=[F^{-1}(\beta_{i})]^{\mu_{i}},
\]
thus $Y_{i}\in[F^{-1}(\beta_{i})]^{\mu_{i}}$. Also for $1\leq i\leq d$ we have
\[
Y_{p+i}^{*}=Z_{i}\subset X_{p+i}\in[F^{-1}(\gamma_{p+i})]^{\kappa_{i}}
\]
and $|Z_{i}|=\mu_{p+i}$, hence $Y_{p+i}^{*}\in[F^{-1}(\beta_{p+i})]^{\mu_{p+i}}$ and
$Y_{p+i}\in[F^{-1}(\beta_{p+i})]^{\mu_{p+i}}$. For the rest of $Y_i$ we have:
\[
Y_{e^{'}+i}=X_{e+i}\subset F^{-1}(\gamma_{e+i})=F^{-1}(\beta_{e+i}),
\]
for $1\leq i <\theta$. Note also that
\[
\theta = \sup\bigr\{|X_{e+i}|;i<\theta\bigr\}=\sup\bigr\{|Y_{e^{'}+i}|;i<\theta\bigr\}.
\]
We establish the inclusion (\ref{hadaf1ib}). Let's first prove that
\begin{equation}\label{hadaf2ib}
\displaystyle{\prod_{i=1}^{e^{'}}}\big{[}Y_{i}\big{]}^{r}\times\big{[}F|(\bigcup_{1\leq i<\theta}Y_{e^{'}+i})\big{]}^{r,f^{'}}\subset S.
\end{equation}
Note that $e=e^{'},f=f^{'}$ and obviously by construction:
\[
\displaystyle{\prod_{i=1}^{e^{'}}}\big{[}Y_{i}\big{]}^{r}\subset\displaystyle{\prod_{i=1}^{e}}\big{[}X_{i}\big{]}^{r},\,\,\,
\big{[}F|(\bigcup_{1\leq i<\theta}Y_{e^{'}\!+i})\big{]}^{r,f^{'}}=\big{[}F|(\bigcup_{1\leq i<\theta}X_{e\!+i})\big{]}^{r,f}.
\]
So (\ref{hadaf2ib}) immediately follow from (\ref{tak}).
In order to prove (\ref{hadaf1ib}) it suffices to show:
\begin{equation}\label{birbirib}
\forall\,\mathbf{x}\in\displaystyle{\prod_{i=1}^{e^{'}}}\big{[}Y_{i}\big{]}^{r}\times\big{[}F|(\bigcup_{1\leq i<\theta}Y_{e^{'}+i})\big{]}^{r,f^{'}}\,\,\, M\models\sigma(\mathbf{x}).
\end{equation}
As in the previous subcase the maximum first index $i$ in the constants $c_{ij}$ occurring in $\sigma$ is $\iota(\sigma)=i_n=e^{'}$, thus it is enough to consider only that part of $\mathbf{x}$ which comes from $[Y_{1}]^{r}\times\dots\times[Y_{e^{'}}]^{r}$, namely
\begin{equation}\label{ikiikiib}
\forall\,\mathbf{x}\in\displaystyle{\prod_{i=1}^{e^{'}}}\big{[}Y_{i}\big{]}^{r}\,\,\,M\models\sigma(\mathbf{x}).
\end{equation}
The rest of the proof of goes the same way as the proof of Subcase (Ia) but with some minor changes. Let $\mathbf{h}$ be an element of $\langle\mathbf{h}_1,\dots,\mathbf{h}_{e^{'}}\rangle\in[Y_1]^{r}\times\dots\times[Y_{e^{'}}]^{r}$. Let $\mathbf{a}=\langle\mathbf{h}_1,\dots,\mathbf{h}_p\rangle$, $\mathbf{b}=\langle\mathbf{h}_{p+1},\dots,\mathbf{h}_{e^{'}}\rangle$. Also for $1\leq i<e^{'}$, set $\mathbf{h}_i=\langle h_{i1},\dots,h_{ir}\rangle$. If $\tau(h_{i_1j_1},\dots,h_{i_nj_n})\leq h_{uv}$, then obviously $M\models\sigma(\mathbf{h})$. So suppose
$\tau(h_{i_1j_1},\dots,h_{i_nj_n})>h_{uv}$.
Then (\ref{ikiikiib}) is reduced to
\begin{equation}\label{uchuchib}
\tau^{\mathrm{left}}(\mathbf{h})=\tau^{\mathrm{right}}(\mathbf{h}).
\end{equation}
Recall that $u<i_{m+1}$, so $u\leq i_{m+1}-1=p$. It follows that $Y_{u}=X_{u}\subset F^{-1}(\gamma_{u})\subset M_{\gamma_p}=M_*$ and consequently $h_{uv}\in Y_u$ is a member of $M_*$. Since we have assumed that $\tau^{\mathrm{left}}(\mathbf{h})<h_{uv}$, it follows that $\tau^{\mathrm{left}}(\mathbf{h})\in M_*$. This will eliminate the possibility (\ref{yekib}). Hence (\ref{doib}) occurs. Thus there is an $\alpha\in M_{\star}$ such that
\begin{equation}\label{taazehib}
\forall\,\mathbf{y}\in\displaystyle{\prod_{i=1}^{d}}\big{[}Y_{p+i}^{*}\big{]}^{r}\,\,\tau^{\mathrm{left}}(\mathbf{a},\mathbf{y})=\alpha.
\end{equation}
Now suppose $\sigma_1,\sigma_2$ are the following two sentences:
\[
\sigma_1:\,\,\,\tau(\underline{h},c_{i_{m+1}j_{m+1}},\dots,c_{i_nj_n})=\alpha,
\]
\[
\sigma_2:\,\,\,\tau(\underline{h},c_{i_{m+1}l_{m+1}},\dots,c_{i_nl_n})=\alpha,
\]
where $\underline{h}=\langle h_{i_1j_1},\dots,h_{i_mj_m}\rangle$. From (\ref{taazehib}), it follows that
\begin{equation}\label{yekyekib}
\forall\,\mathbf{y}\in\displaystyle{\prod_{i=p+1}^{e^{'}}}\big{[}Y_i^*\big{]}^r\,\,\,M\models\sigma_1(\mathbf{a},\mathbf{y}).
\end{equation}
But the two sequences $\langle c_{i_{m+1}j_{m+1}},\dots,c_{i_nj_n}\rangle$, $\langle c_{i_{m+1}l_{m+1}},\dots,c_{i_nl_n}\rangle$
are equivalent and hence Lemma \ref{combinatoriallemma} would imply
\begin{equation}\label{dodoib}
\forall\,\mathbf{y}\in\displaystyle{\prod_{i=p+1}^{e^{'}}}\big{[}Y_i\big{]}^r\,\,\,M\models\sigma_2(\mathbf{a},\mathbf{y}).
\end{equation}
Putting (\ref{yekyekib}) and (\ref{dodoib}) together we obtain
\[
\forall\,\mathbf{y}\in\displaystyle{\prod_{i=p+1}^{e^{'}}}\big{[}Y_i\big{]}^r\,\,\,
\tau^{\mathrm{left}}(\mathbf{a},\mathbf{y})=\tau^{\mathrm{right}}(\mathbf{a},\mathbf{y}),
\]
which implies that $\tau^{\mathrm{left}}(\mathbf{a},\mathbf{b})=\tau^{\mathrm{right}}(\mathbf{a},\mathbf{b})$ and consequently $\tau^{\mathrm{left}}(\mathbf{h})=\tau^{\mathrm{right}}(\mathbf{h})$. This confirms (\ref{uchuchib}), hence the proof of Subcase (Ib).\\

\underline{\underline{Subcase (Ic)}}: ($p<e<e^{'}$)\\

Let $p+d=e, e+d^{'}=e^{'}$. For $1\leq i\leq d+d^{'}$, define cardinals $\kappa_{i},\lambda_{i}$ as follows: If $i=1$, then $\lambda_{1}=\max(\mu_{p+1},2^{\chi})$, $\beth_{r-1}(\lambda_{1})<\kappa_{1}<\theta$ and if $i>1$, then $\lambda_{i}=\max(\mu_{p+i},\kappa_{i-1})$, $\beth_{r-1}(\lambda_{i})<\kappa_{i}<\theta$. Having in mind the strategy of the player I in the game $G(S^{'},e^{'}):$
\[
\mu_1,\dots,\mu_p,\mu_{p+1},\dots,\mu_e,\mu_{e+1},\mu_{e^{'}}.
\]
Suppose that the player I plays the following strategy in the game $G(S,e)$:
\[
\mu_1,\dots,\mu_p,\kappa_{1},\dots,\kappa_{d}.
\]
Then the player II would play the game if he plays according to his winning strategy in the game $G(S,e)$. Suppose the move are
\[
\gamma_1,\dots,\gamma_p,\gamma_{p+1},\dots,\gamma_{p+d},\gamma_{e+1},\gamma_{e+i},\dots\,\,\, i<\theta
\]
Thus the above sequence is strictly increasing and there are sets
\begin{center}
$X_{1}\in\big{[}F^{-1}(\gamma_{1})\big{]}^{\mu_{1}},\dots,X_{p}\in\big{[}F^{-1}(\gamma_{p})\big{]}^{\mu_{p}},$
\end{center}
\begin{center}
$X_{p+1}\in\big{[}F^{-1}(\gamma_{p+1})\big{]}^{\kappa_{1}},\dots,X_{p+d}\in\big{[}F^{-1}(\gamma_{p+d})\big{]}^{\kappa_{d}}$
\end{center}
as well as the sets
\begin{center}
$X_{e+i}\subset F^{-1}(\gamma_{e+i})$ \,\,\, for \,\,\, $1\leq i<\theta$
\end{center}
such that
\begin{equation}\label{1ic}
\sup\bigr\{|X_{e+i}|;i<\theta\bigr\}=\theta
\end{equation}
and
\begin{equation}\label{*ic}
\displaystyle{\prod_{i=1}^{e}}\big{[}X_{i}\big{]}^{r}\times\big{[}F|(\displaystyle{\bigcup_{1\leq i<\theta}}X_{e+i})\big{]}^{r,f}\subset S.
\end{equation}
Now we are ready to define $\beta_{i}$. Set
\begin{eqnarray*}
&&\beta_{i}(\mu_{1},\dots,\mu_{i}) = \gamma_{i}(\mu_{1},\dots,\mu_{i})\,\,\,\mathrm{for}\,\,\,1\leq i\leq p,\\
&&\beta_{p+i}(\mu_{1},\dots,\mu_{p+i}) = \gamma_{p+i}(\mu_{1},\dots,\mu_{p},\kappa_1,\dots,\kappa_i)\,\,\,\mathrm{for}\,\,\,1\leq i\leq d.
\end{eqnarray*}
In order to define
\[
\beta_{e+1},\dots,\beta_{e+d^{'}},\beta_{e^{'}+1},\dots,\beta_{e^{'}+i},\dots\,\,\, i<\theta
\]
we need to introduce ordinals $\delta_1,\dots,\delta_{d^{'}}<\theta$ such that $\delta_1$ is the least ordinal $<\theta$ such that $|X_{e+\delta_{1}}|<\kappa_{e+1}$ and if $d^{'}\geq 2$, then for $2\leq i\leq d^{'}$ let $\delta_i$ be the least ordinal $<\theta$ such that $\delta_i>\delta_{i-1}$ and $|X_{e+\delta_{i}}|\geq\kappa_{e+i}$. This is possible because of (\ref{1ic}). Now set
\begin{eqnarray*}
&&\beta_{e+i}(\mu_{1},\dots,\mu_{e+i}) = \gamma_{e+\delta_i}(\mu_{1},\dots,\mu_{p},\kappa_1,\dots,\kappa_d)\,\,\,\mathrm{for}\,\,\,1\leq i\leq d^{'},\\
&&\beta_{e^{'}+i}(\mu_{1},\dots,\mu_{e^{'}}) = \gamma_{e+\delta_{d^{'}}+i}(\mu_{1},\dots,\mu_{p},\kappa_1,\dots,\kappa_d)\,\,\,\mathrm{for}\,\,\,1\leq i<\theta.
\end{eqnarray*}
this completes the description of the strategy of the player II for the game $G(S^{'},e^{'})$. We shall prove that it is a winning strategy. By our choice of $\beta_i$ it is evident that
\[
\beta_1<\beta_2<\dots<\beta_{e^{'}}<\beta_{e^{'}+1}<\dots<\beta_{e^{'}+i}<\dots\,\,\,\,\,i<\theta.
\]
We must find $Y_i$'s such that
\begin{equation}\label{3ic}
Y_{1}\in\big{[}F^{-1}(\beta_{1})\big{]}^{\mu_{1}},\dots,Y_{e^{'}}\in\big{[}F^{-1}(\beta_{e^{'}})\big{]}^{\mu_{e^{'}}}
\end{equation}
as well as
\begin{equation}\label{4ic}
Y_{e^{'}+i}\subset F^{-1}(\beta_{e^{'}+i})
\end{equation}
for $1\leq i<\theta$ where
\begin{equation}\label{5ic}
\sup\bigr\{|Y_{e^{'}+i}|;1\leq i<\theta\bigr\}=\theta
\end{equation}
and
\begin{equation}\label{6ic}
\big{[}Y_{1}\big{]}^{r}\times\dots\times\big{[}Y_{e^{'}}\big{]}^{r}\times\big{[}F|(\displaystyle{\bigcup_{1\leq i<\theta}}Y_{e^{'}+i})\big{]}^{r,f^{'}}\subset S^{'},
\end{equation}
where $s-e^{'}=f^{'}\geq 0$. As in the previous subcases it is time to enter the Erd\"{o}s and Rado's polarized partition relation into the scene. By Corollary \ref{usefulcorollary} we have
\begin{equation}\label{7ic}
(\kappa_1,\dots,\kappa_d,\dots,\kappa_{d+d^{'}})\longrightarrow(\mu_{p+1},\dots,\mu_e,\dots,\mu_{e^{'}})^{r}_{2^{\chi}}.
\end{equation}
We shall introduce a partition relation $\mathcal{R}$ on the set
\[
[X_{p+1}]^{r}\times\dots\times[X_{p+d}]^{r}\times[X_{e+\delta_1}]^{r}\times\dots\times[X_{e+\delta_{d^{'}}}]^{r}
\]
as follows: For any $\alpha\in M_{*}\cup\{\star\}$ and any $\mathbf{a}\in[X_{1}]^{r}\times\dots\times[X_{p}]^{r}$, let
\[
P_{\alpha,\mathbf{a}}=\bigr\{\mathbf{x}\in\displaystyle{\prod_{i=1}^{d}}[X_{p+i}]^{r}\times\displaystyle{\prod_{i=1}^{d^{'}}}[X_{e+\delta_{i}}]^{r}; \tau^{\mathrm{left}}(\mathbf{a},\mathbf{x})=\alpha\bigr\}
\]
where $\tau^{\mathrm{left}}(\mathbf{a},\mathbf{x})=\star$ is an abbreviation for $\tau^{\mathrm{left}}(\mathbf{a},\mathbf{x})>M_{*}$. for any $\mathbf{a}$ as above, the set $\bigr\{P_{\alpha,\mathbf{a}};\alpha\in M_{*}\cup\{\star\}\bigr\}$ forms a partition for the set
\[
\displaystyle{\prod_{i=1}^{d}}[X_{p+i}]^{r}\times\displaystyle{\prod_{i=1}^{d^{'}}}[X_{e+\delta_{i}}]^{r}
\]
which we denote by $\mathcal{R}_{\mathbf{a}}$. Let $\mathcal{R}$ be a partition relation such that
\begin{center}
$\forall\,\mathbf{x}_1,\mathbf{x}_2\in\displaystyle{\prod_{i=1}^{d}}[X_{p+i}]^{r}\times\displaystyle{\prod_{i=1}^{d^{'}}}[X_{e+\delta_{i}}]^{r}:
\mathbf{x}_1\mathcal{R}\mathbf{x}_2$  iff  $\forall\,{\mathbf{a}}\in\displaystyle{\prod_{i=1}^{p}}[X_{i}]^{r}\mathbf{x}_1\mathcal{R}_{\mathbf{a}}\mathbf{x}_2$.
\end{center}
The number of partition classes is at most $2^{\chi}$. Hence by (\ref{7ic}) there are subsets $Z_{i}\subset X_{p+i}$ for $1\leq i\leq d$ such that $|Z_i|=\mu_{p+i}$ and also subset $Z_{d+i}\subset X_{e+\delta_{i}}$ for $1\leq i\leq d^{'}$ such that $|Z_{d+i}|=\mu_{e+i}$ and the set
\[
[Z_1]^{r}\times\dots\times[Z_d]^{r}\times\dots\times[Z_{d+d^{'}}]^{r}
\]
lies in one partition class. Now for $1\leq i\leq p$ put $Y_i^{*}=X_i$ and for $1\leq i\leq d+d^{'}$ put $Y_{p+i}^{*}=Z_i$. Also let $Y_{e^{'}+i}=X_{e+\delta_{d^{'}}+i}$ for $1\leq i <\theta$. Finally for $1\leq i<e^{'}$ let $Y_i$ be any member of $(Y_{i}^{*})^{\bullet\bullet}$ in the sense of Fact \ref{fact}. Now we can deduce that
\[
\forall\,\mathbf{a}\in\displaystyle{\prod_{i=1}^{p}}[Y_{i}^{*}]^{r}
\]
\underline{either}
\[
\forall\,\mathbf{x}\in\displaystyle{\prod_{i=1}^{d+d^{'}}}[Y_{p+i}^{*}]^{r}\,\,\tau^{\mathrm{left}}(\mathbf{a},\mathbf{x})>M_{*},
\]
\underline{or} there exists $\alpha\in M_{*}$ such that
\begin{equation}\label{11ic}
\forall\,\mathbf{x}\in\displaystyle{\prod_{i=1}^{d+d^{'}}}[Y_{p+i}^{*}]^{r}\,\,\tau^{\mathrm{left}}(\mathbf{a},\mathbf{x})=\alpha.
\end{equation}
The next step is verifying that the required properties (\ref{3ic}), (\ref{4ic}), (\ref{5ic}) and (\ref{6ic}) of $Y_i$ hold. Of course for $1\leq i\leq p$ we have \[
Y_i^{*}=X_i\in[f^{-1}(\gamma_i)]^{\mu_i}=[f^{-1}(\beta_i)]^{\mu_i}.
\]
Thus $Y_i\in[f^{-1}(\beta_i)]^{\mu_i}$. Also for $1\leq i\leq d$ we have
\[
Y_{p+i}^{*}=Z_i\subset X_{p+i}\in[F^{-1}(\gamma_{p+i})]^{\kappa_{i}}=[F^{-1}(\beta_{p+i})]^{\kappa_{i}}
\]
and $|Z_i|=\mu_{p+i}$, hence $Y_{p+i}^{*}\in[F^{-1}(\beta_{p+i})]^{\mu_{p+i}}$, so $Y_{p+i}\in[F^{-1}(\beta_{p+i})]^{\mu_{p+i}}$.

\noindent For $1\leq i\leq d^{'}$ we have
\[
Y_{e+i}^{*}=Z_{d+i}\subset X_{e+\delta_{i}}\in F^{-1}(\gamma_{e+\delta_{i}})=F^{-1}(\beta_{e+i})
\]
with $|Z_{d+i}|=\mu_{e+i}$, so $Y_{e+i}^{*}\in[F^{-1}(\beta_{e+i})]^{\mu_{e+i}}$, hence $Y_{e+i}\in[F^{-1}(\beta_{e+i})]^{\mu_{e+i}}$.

\noindent Finally, for $1\leq i<\theta$:
\[
Y_{e^{'}+i}=X_{e+\delta_{d^{'}}+i}\subset F^{-1}(\gamma_{e+\delta_{d^{'}}+i})=F^{-1}(\beta_{e^{'}+i}).
\]
It is easy to see that $\sup\bigr\{|Y_{e^{'}+i}|;i<\theta\bigr\}=\sup\bigr\{|X_{e+\delta_{d^{'}}+i}|;i<\theta\bigr\}=\theta$. Now it remains to prove (\ref{6ic}). As in the previous cases we begin with stating that
\begin{equation}\label{10ic}
\displaystyle{\prod_{i=1}^{e^{'}}}\big{[}Y_i\big{]}^{r}\times \big{[}F|(\displaystyle{\bigcup_{1\leq i<\theta}}Y_{e^{'}+i})\big{]}^{r,f^{'}}\subset S.
\end{equation}
[Why?  obviously
\begin{equation}\label{7ic}
\displaystyle{\prod_{i=1}^{e^{'}}}\big{[}Y_i\big{]}^{r}=\displaystyle{\prod_{i=1}^{e}}\big{[}Y_i\big{]}^{r}\times
\displaystyle{\prod_{i=e+1}^{e^{'}}}\big{[}Y_i\big{]}^{r}\subset\displaystyle{\prod_{i=1}^{e}}\big{[}X_i\big{]}^{r}\times
\displaystyle{\prod_{i=1}^{d^{'}}}\big{[}X_{e+\delta_{i}}\big{]}^{r}
\end{equation}
and
\begin{equation}\label{8ic}
\big{[}F|(\displaystyle{\bigcup_{1\leq i<\theta}}Y_{e^{'}+i})\big{]}^{r,f^{'}}\subset\big{[}F|(\displaystyle{\bigcup_{1\leq i<\theta}}X_{e+\delta_d+i})\big{]}^{r,f^{'}}.
\end{equation}
Recall that $e+d^{'}=e^{'}$, so $f=f^{'}+d^{'}$. It is also clear that
\begin{equation}\label{9ic}
\displaystyle{\prod_{i=1}^{d^{'}}}\big{[}X_{e+\delta_{i}}\big{]}^{r}\times\big{[}F|(\displaystyle{\bigcup_{1\leq i<\theta}}X_{e+\delta_d+i})\big{]}^{r,f^{'}}\subset\big{[}F|(\displaystyle{\bigcup_{1\leq i<\theta}}X_{e+i})\big{]}^{r,f}.
\end{equation}
Therefore (\ref{7ic}), (\ref{8ic}) and (\ref{9ic}) imply that
\[
\displaystyle{\prod_{i=1}^{e^{'}}}\big{[}Y_{i}\big{]}^{r}\times\big{[}F|(\displaystyle{\bigcup_{1\leq i<\theta}}Y_{e^{'}+i})\big{]}^{r,f^{'}}\subset
\displaystyle{\prod_{i=1}^{e}}\big{[}X_{i}\big{]}^{r}\times\big{[}F|(\displaystyle{\bigcup_{1\leq i<\theta}}X_{e+i})\big{]}^{r,f}.
\]
So (\ref{10ic}) immediately follows from (\ref{*ic}).]

We shall complete the proof of (\ref{6ic}) by showing that
\[
\forall\,\mathbf{x}\in\displaystyle{\prod_{i=1}^{e^{'}}}\big{[}Y_{i}\big{]}^{r}\times
\big{[}F|(\displaystyle{\bigcup_{1\leq i<\theta}}Y_{e^{'}+i})\big{]}^{r,f^{'}} M\models\sigma(\mathbf{x}).
\]
Since $\iota(\sigma)=i_n=e^{'}$ it is sufficient to establish
\[
\forall\,\mathbf{x}\in\displaystyle{\prod_{i=1}^{e^{'}}}\big{[}Y_{i}\big{]}^{r} M\models\sigma(\mathbf{x}).
\]
Let $\mathbf{h}=\langle\mathbf{h}_1,\dots,\mathbf{h}_{e^{'}}\rangle\in[Y_1]^{r}\times\dots\times[Y_{e^{'}}]^{r}$, $\mathbf{a}=\langle\mathbf{h}_1,\dots,\mathbf{h}_{p}\rangle$, $\mathbf{b}=\langle\mathbf{h}_{p+1},\dots,\mathbf{h}_{e^{'}}\rangle$. So $\mathbf{h}=\langle\mathbf{a},\mathbf{b}\rangle$. We intend to show $M\models\sigma(\mathbf{h})$. For $1\leq i\leq e^{'}$, put $\mathbf{h}_i=\langle h_{i1},\dots, h_{ir}\rangle$. If $\tau(h_{i_1j_1},\dots,h_{i_nj_n})\leq h_{uv}$, then automatically $M\models\sigma(\mathbf{h})$. So suppose $\tau(h_{i_1j_1},\dots,h_{i_nj_n})> h_{uv}$. In this case $M\models\sigma(\mathbf{h})$ is equivalent to
\[
M\models\tau^{\mathrm{left}}(\mathbf{a},\mathbf{b})=\tau^{\mathrm{right}}(\mathbf{a},\mathbf{b}).
\]
But $h_{uv}\in M_{*}$ and then $\tau^{\mathrm{left}}(\mathbf{h})\in M_{*}$, so by (\ref{11ic}) we have
\begin{equation}\label{12ic}
\forall\,\mathbf{y}\in\displaystyle{\prod_{i=1}^{d+d^{'}}}\big{[}Y^{*}_{p+i}\big{]}^{r}\tau^{\mathrm{left}}(\mathbf{a},\mathbf{y})=\alpha.
\end{equation}
If $\sigma_1,\sigma_2$ are the following two sentences
\[
\sigma_1:\tau(\underline{h},c_{i_{m+1}j_{m+1}},\dots,c_{i_{n}j_{n}})=\alpha,
\]
\[
\sigma_2:\tau(\underline{h},c_{i_{m+1}l_{m+1}},\dots,c_{i_{n}l_{n}})=\alpha
\]
where $\underline{h}=\langle h_{i_1j_1},\dots,h_{i_mj_m}\rangle$, then (\ref{12ic}) implies that
\begin{equation}\label{13ic}
\forall\,\mathbf{y}\in\displaystyle{\prod_{i=1}^{d+d^{'}}}\big{[}Y^{*}_{p+i}\big{]}^{r}M\models\sigma_1(\mathbf{a},\mathbf{y}).
\end{equation}
Also from the equivalence of $\langle c_{i_{m+1}j_{m+1}},\dots,c_{i_{n}j_{n}}\rangle$ and $\langle c_{i_{m+1}l_{m+1}},\dots,c_{i_{n}l_{n}}\rangle$, along with Lemma \ref{combinatoriallemma}, we conclude that
\begin{equation}\label{13ic}
\forall\,\mathbf{y}\in\displaystyle{\prod_{i=1}^{d+d^{'}}}\big{[}Y_{p+i}\big{]}^{r}M\models\sigma_2(\mathbf{a},\mathbf{y}).
\end{equation}
Now (\ref{13ic}), (\ref{13ic}) would reveal that
\[
\forall\,\mathbf{y}\in\displaystyle{\prod_{i=1}^{d+d^{'}}}\big{[}Y_{p+i}\big{]}^{r}
M\models\tau^{\mathrm{left}}(\mathbf{a},\mathbf{y})=\tau^{\mathrm{right}}(\mathbf{a},\mathbf{y}).
\]
which implies that $M\models\sigma(\mathbf{a},\mathbf{b})$, hence the proof of Subcase (Ic).\\

\underline{Case II}: ($e<p$)\\

Let $e+d=p^{'},p+d^{'}=e^{'}$, where $d,d^{'}>0$. Recall the strategy of the player I:
\[
\mu_1,\mu_2,\dots,\mu_{e^{'}}
\]
for the game $G(S^{'},e^{'})$ and also recall the winning strategy of the strategy of the player II for the game $G(S,e)$:
\[
\gamma_1,\dots,\gamma_e,\gamma_{e+1},\dots,\gamma_{e+i},\dots\,\,\,i<\theta
\]
So if we assume
\[
\gamma_{i}=\gamma_{i}(\mu_1,\dots,\mu_i)\,\,\,\mathrm{for}\,\,\, 1\leq i\leq e,
\]
\[
\gamma_{e+i}=\gamma_{e+i}(\mu_1,\dots,\mu_e)\,\,\,\mathrm{for}\,\,\, 1\leq i<\theta,
\]
then there are sets:
\[
X_1\in[F^{-1}(\gamma_1)]^{\mu_1},\dots,X_e\in[F^{-1}(\gamma_e)]^{\mu_e},
\]
\[
X_{e+i}\subset F^{-1}(\gamma_{e+i})
\]
with
\[
\sup\bigr\{|X_{e+i}|;i<\theta\bigr\}=\theta
\]
such that
\[
\displaystyle{\prod_{i=1}^{e}}\big{[}X_i\big{]}^{r}\times \big{[}F|(\displaystyle{\bigcup_{1\leq i<\theta}}X_{e+i})\big{]}^{r,f}\subset S
\]
where $s-e=f$. Now set
\[
\beta_{i}(\mu_1,\dots,\mu_i)=\gamma_{i}(\mu_1,\dots,\mu_i)\,\,\,\mathrm{for}\,\,\, 1\leq i\leq e.
\]
For $1\leq i\leq d$, let $\delta_{i}$ be the least ordinal $<\theta$ such that there is $X_{e+\delta_{i}}\subset F^{-1}(\gamma_{e+\delta_i})$ with $|X_{e+\delta_{i}}|\geq\mu_{e+i}$. We additionally may suppose that $\delta_1<\delta_2<\dots<\delta_d$. Also set
\[
\beta_{e+i}(\mu_1,\dots,\mu_e,\dots,\mu_{e+i})=\gamma_{e+\delta_i}(\mu_1,\dots,\mu_e)\,\,\,\mathrm{for}\,\,\,1\leq i\leq d.
\]
We need to set up the situation before defining the rest of $\beta_i$. This will be done by employing the Erd\"{o}s-Rado polarized partition theorem. Assume that $M_{*}=M_{e+\delta_{d}}$ and $\star$ is a symbol different from all elements of $M$. let $\chi$ denotes the cardinality of $M_{*}$. Now for $1\leq i\leq d^{'}$ define the cardinals $\kappa_i,\lambda_i$ as follows: If $i=1$, then $\lambda_1=\max(\mu_{p+1},2^{\chi})$, $\beth_{r-1}(\lambda_1)<\kappa_1<\theta$. If $i>1$, then $\lambda_i=\max(\mu_{p+i},\kappa_{i-1})$, $\beth_{r-1}(\lambda_i)<\kappa_i<\theta$. By Corollary \ref{usefulcorollary} we have
\begin{equation}\label{1caseii}
(\kappa_1,\dots,\kappa_{d^{'}})\longrightarrow(\mu_{p+1},\dots,\mu_{e^{'}})_{2^{\chi}}^{r}
\end{equation}
Now for $1\leq i\leq d^{'}$, let $\delta_{d+i}$ be the least ordinal $<\theta$ such that $\delta_{d+i}>\delta_{d+i-1}$ and there is $X_{e+\delta_{d+i}}\subset F^{-1}(\gamma_{e+\delta_{d+i}})$ with $|X_{e+\delta_{d+i}}|\geq\kappa_i$. Set
\[
\beta_{p+i}(\mu_1,\dots,\mu_{p+i})=\gamma_{e+\delta_{d+i}}(\mu_1,\dots,\mu_e)\,\,\,\,\mathrm{for}\,\,\,1\leq i\leq d^{'}.
\]
Also set
\[
\beta_{e^{'}+i}(\mu_1,\dots,\mu_{p+i})=\gamma_{e+\delta_{d+d^{'}}+i}(\mu_1,\dots,\mu_e)\,\,\,\,\mathrm{for}\,\,\,1\leq i<\theta.
\]
We claim that the strategy $\beta_i$ defined above constitutes a winning strategy for the player II in the game $G(S^{'},e^{'})$. Clearly it gives a strictly increasing sequence of moves for the player II. We shall prove that there are sets
\begin{equation}\label{2caseii}
Y_1\in[F^{-1}(\beta_1)]^{\mu_1},\dots,Y_{e^{'}}\in[F^{-1}(\beta_{e^{'}})]^{\mu_{e^{'}}}
\end{equation}
\begin{equation}\label{3caseii}
Y_{e^{'}+i}\subset F^{-1}(\beta_{e^{'}+i})\,\,\,\mathrm{for}\,\,\, 1\leq i<\theta
\end{equation}
such that
\begin{equation}\label{4caseii}
\sup\bigr\{|Y_{e^{'}+i}|;i<\theta\bigr\}=\theta
\end{equation}
and
\begin{equation}\label{5caseii}
\displaystyle{\prod_{i=1}^{e^{'}}}\big{[}Y_i\big{]}^{r}\times \big{[}F|(\displaystyle{\bigcup_{1\leq i<\theta}}Y_{e^{'}+i})\big{]}^{r,f^{'}}\subset S^{'}
\end{equation}
where $s-e^{'}=f^{'}$. For any $\alpha\in M_{*}\cup\{\star\}$ and any
\[
\mathbf{a}\in[X_1]^{r}\times\dots\times[X_e]^{r}\times[X_{e+\delta_{1}}]^{r}\times\dots\times[X_{e+\delta_{d}}]^{r}
\]
let
\[
P_{\alpha,\mathbf{a}}=\bigr\{\mathbf{x}\in\displaystyle{\prod_{i=1}^{d^{'}}}[X_{e+\delta_{d+i}}]^{r};\tau^{\mathrm{left}}(\mathbf{a},\mathbf{x})=\alpha\bigr\}.
\]
As usual $\tau^{\mathrm{left}}(\mathbf{a},\mathbf{x})=\star$ is an abbreviation for $\tau^{\mathrm{left}}(\mathbf{a},\mathbf{x})>M_{\star}$. Fixing any $\mathbf{a}$ as above, the set $\bigr\{P_{\alpha,\mathbf{a}}|\alpha\in M_{*}\cup\{\star\}\bigr\}$ becomes a partition for the set
\[
[X_{e+\delta_{d+1}}]^{r}\times\dots\times[X_{e+\delta_{d+d^{'}}}]^{r}.
\]
We denote the partition relation by $\mathcal{R}_{\mathbf{a}}$. Then the desired $\mathcal{R}$ would be defined as
\begin{center}
$\forall\,\mathbf{x}_1,\mathbf{x}_2\in\displaystyle{\prod_{i=1}^{d^{'}}}\big{[}X_{e+\delta_{d+i}}\big{]}^{r}:\mathbf{x}_1\mathcal{R}\mathbf{x}_2$\,\,\,iff\,\,\,
$\forall\,\mathbf{a}\in\displaystyle{\prod_{i=1}^{e}}\big{[}X_{i}\big{]}^{r}\times\displaystyle{\prod_{i=1}^{d}}\big{[}X_{e+\delta_{i}}\big{]}^{r}
\mathbf{x}_1\mathcal{R}_{\mathbf{a}}\mathbf{x}_2$.
\end{center}
The number of the partition classes is at most $2^{\chi}$. Hence by (\ref{1caseii}), there are subsets $Z_i\subset X_{e+\delta_{d+i}}$ for $1\leq i\leq d^{'}$ such that $|Z_{i}|=\mu_{p+i}$ and the following set lies in one partition class:
\[
[Z_1]^{r}\times\dots\times[Z_{d^{'}}]^{r}.
\]
Now set
\[
Y_i^{*}= X_i\,\,\,\mathrm{for}\,\,\, 1\leq i\leq p,
\]
\[
Y_{p+i}^{*}= Z_{i}\,\,\,\mathrm{for}\,\,\, 1\leq i\leq d^{'},
\]
\[
Y_{e^{'}+i}= X_{e+\delta_{d+d^{'}}+i}\,\,\,\mathrm{for}\,\,\, 1\leq i<\theta.
\]
Finally for $1\leq i\leq e^{'}$, let $Y_i$ be an arbitrary element of $(Y_i^{*})^{\bullet\bullet}$. Now for every $\mathbf{a}$ from $[Y_1^{*}]^{r}\times\dots\times[Y^{*}_p]^{r}$ we have \underline{either}
\[
\forall\,\mathbf{x}\in\displaystyle{\prod_{i=1}^{d+d^{'}}}\big{[}Y_{e+i}^{*}\big{]}^{r}\,\,\, \tau^{\mathrm{left}}(\mathbf{a},\mathbf{x})>M_{*},
\]
\underline{or} there exists $\alpha\in M_{*}$ such that
\begin{equation}\label{o-o}
\forall\,\mathbf{x}\in\displaystyle{\prod_{i=1}^{d+d^{'}}}\big{[}Y_{e+i}^{*}\big{]}^{r}\,\,\, \tau^{\mathrm{left}}(\mathbf{a},\mathbf{x})=\alpha.
\end{equation}
We show that $Y_i$ satisfy the relations (\ref{2caseii}) through (\ref{5caseii}). If $1\leq i\leq p$, then
\[
Y_i^{*}=X_i\in[F^{-1}(\gamma_i)]^{\mu_i}=[F^{-1}(\beta_i)]^{\mu_i},
\]
so $Y_i\in[F^{-1}(\beta_i)]^{\mu_i}$. For $1\leq i\leq d$:
\[
Y_{e+i}^{*}=X_{e+\delta_{i}}\in[F^{-1}(\gamma_{e+\delta_{i}})]^{\mu_{e+i}}=[F^{-1}(\beta_{e+i})]^{\mu_{e+i}}.
\]
Thus $Y_{e+i}\in[F^{-1}(\beta_{e+i})]^{\mu_{e+i}}$. If $1\leq i\leq d^{'}$, then
\[
Y_{p+i}^{*}=Z_i\subset X_{e+\delta_{d+i}}\in[F^{-1}(\gamma_{e+\delta_{d+i}})]^{\kappa_{i}}=[F^{-1}(\beta_{p+i})]^{\kappa_{i}},
\]
but $|Z_i|=\mu_{p+i}$ , hence $Y_{p+i}^{*}\in[F^{-1}(\beta_{p+i})]^{\mu_{p+i}}$ and immediately $Y_{p+i}\in[F^{-1}(\beta_{p+i})]^{\mu_{p+i}}$. This proves (\ref{2caseii}). Also for $1\leq i<\theta$:
\[
Y_{e^{'}+i}=X_{e+\delta_{d+d^{'}}+i}\subset F^{-1}(\gamma_{e+\delta_{d+d^{'}}+i})=F^{-1}(\beta_{e^{'}+i}),
\]
which proves (\ref{3caseii}). Obviously $\sup\bigr\{|Y_{e^{'}+i}|;i<\theta\bigr\}=\sup\bigr\{|X_{e+\delta_{d+d^{'}}+i}|;i<\theta\bigr\}=\theta$. So we have (\ref{4caseii}). It remains to prove (\ref{5caseii}). As in the previous cases we start with claiming that
\begin{equation}\label{golcaseii}
\displaystyle{\prod_{i=1}^{e^{'}}}\big{[}Y_i\big{]}^{r}\times \big{[}F|(\displaystyle{\bigcup_{1\leq i<\theta}}Y_{e^{'}+i})\big{]}^{r,f^{'}}\subset S.
\end{equation}
[Why? Observe that the left side of the above relation can be written as
\[
\displaystyle{\prod_{i=1}^{e}}\big{[}Y_i\big{]}^{r}\times
\displaystyle{\prod_{i=1}^{d}}\big{[}Y_{e+i}\big{]}^{r}\times
\displaystyle{\prod_{i=1}^{d^{'}}}\big{[}Y_{p+i}\big{]}^{r}\times
\big{[}F|(\displaystyle{\bigcup_{1\leq i<\theta}}Y_{e^{'}+i})\big{]}^{r,f^{'}}.
\]
By construction
\[
\displaystyle{\prod_{i=1}^{e}}\big{[}Y_i\big{]}^{r}=
\displaystyle{\prod_{i=1}^{e}}\big{[}X_i\big{]}^{r},\,\,\,
\displaystyle{\prod_{i=1}^{d}}\big{[}Y_{e+i}\big{]}^{r}\subset
\displaystyle{\prod_{i=1}^{d}}\big{[}X_{e+\delta_i}\big{]}^{r},\,\,\,
\displaystyle{\prod_{i=1}^{d^{'}}}\big{[}Y_{p+i}\big{]}^{r}\subset
\displaystyle{\prod_{i=1}^{d^{'}}}\big{[}X_{e+\delta_{d+i}}\big{]}^{r}
\]
as well as
\[
\big{[}F|(\displaystyle{\bigcup_{1\leq i<\theta}}Y_{e^{'}+i})\big{]}^{r,f^{'}}\subset
\big{[}F|(\displaystyle{\bigcup_{1\leq i<\theta}}X_{e+\delta_{d+d^{'}}+i})\big{]}^{r,f^{'}}.
\]
Since $f^{'}+d+d^{'}=f$, we can conclude that
\[
\displaystyle{\prod_{i=1}^{d}}\big{[}X_{e+\delta_i}\big{]}^{r}\times
\displaystyle{\prod_{i=1}^{d^{'}}}\big{[}X_{e+\delta_{d+i}}\big{]}^{r}\times
\big{[}F|(\displaystyle{\bigcup_{1\leq i<\theta}}X_{e+\delta_{d+d^{'}}+i})\big{]}^{r,f^{'}}\subset
\big{[}F|(\displaystyle{\bigcup_{1\leq i<\theta}}X_{e+i})\big{]}^{r,f}.
\]
Therefore
\[
\displaystyle{\prod_{i=1}^{e^{'}}}\big{[}Y_i\big{]}^{r}\times \big{[}F|(\displaystyle{\bigcup_{1\leq i<\theta}}Y_{e^{'}+i})\big{]}^{r,f^{'}}\subset
\displaystyle{\prod_{i=1}^{e}}\big{[}X_i\big{]}^{r}\times \big{[}F|(\displaystyle{\bigcup_{1\leq i<\theta}}X_{e+i})\big{]}^{r,f}\subset S,
\]
which proves (\ref{golcaseii}).]

For the last step of establishing Case II we must show that
\[
\forall\,\mathbf{x}\in\displaystyle{\prod_{i=1}^{e^{'}}}\big{[}Y_i\big{]}^{r}\times \big{[}F|(\displaystyle{\bigcup_{1\leq i<\theta}}Y_{e^{'}+i})\big{]}^{r,f^{'}}\,\,\,M\models\sigma(\mathbf{x}).
\]
Since $\iota(\sigma)=i_n=e^{'}$, it reduces to show
\[
\forall\,\mathbf{x}\in\displaystyle{\prod_{i=1}^{e^{'}}}\big{[}Y_i\big{]}^{r}\,\,\,M\models\sigma(\mathbf{x}).
\]
Choose an element $\mathbf{h}=\langle\mathbf{h}_1,\dots,\mathbf{h}_{e^{'}}\rangle\in[Y_1]^{r}\times\dots[Y_{e^{'}}]^r$ and let $\mathbf{a}=\langle\mathbf{h}_1,\dots,\mathbf{h}_{p}\rangle$ and $\mathbf{b}=\langle\mathbf{h}_{p+1},\dots,\mathbf{h}_{e^{'}}\rangle$. So $\mathbf{h}=\langle\mathbf{a},\mathbf{b}\rangle$. Let $\mathbf{h}_i=\langle h_{i1},\dots,h_{ir}\rangle$. If $\tau(h_{i_1j_1},\dots,h_{i_nj_n})\leq h_{uv}$, then we get $M\models\sigma(\mathbf{h})$. So suppose that $\tau(h_{i_1j_1},\dots,h_{i_nj_n})> h_{uv}$. The assertion $M\models\sigma(\mathbf{h})$ is equivalent to
\[
M\models\tau^{\mathrm{left}}(\mathbf{a},\mathbf{b})=\tau^{\mathrm{right}}(\mathbf{a},\mathbf{b}).
\]
Observe that $u\leq i_n-1=p$ and $h_{uv}\in Y_u$. But
\[
Y_u\subset\displaystyle{\bigcup_{i=1}^{e}}X_i\cup\displaystyle{\bigcup_{i=1}^{d}}X_{e+\delta_i}\subset
\displaystyle{\bigcup_{i=1}^{e}}F^{-1}(\gamma_i)\cup\displaystyle{\bigcup_{i=1}^{d}}F^{-1}(\gamma_{e+\delta_i})\subset
M_{e+\delta_d}=M_*.
\]
Hence $h_{uv}\in M_*$. So $\tau^{\mathrm{left}}(\mathbf{h})\in M_*$. Now by (\ref{o-o}) we have
\begin{equation}\label{12caseii}
\forall\,\mathbf{y}\in\displaystyle{\prod_{i=1}^{d^{'}}}\big{[}Y^{*}_{p+i}\big{]}^r\,\,\,\tau^{\mathrm{left}}(\mathbf{a},\mathbf{y})=\alpha.
\end{equation}
Set
\[
\sigma_1:\tau(\underline{h},c_{i_{m+1}j_{m+1}},\dots,c_{i_nj_n})=\alpha,
\]
\[
\sigma_2:\tau(\underline{h},c_{i_{m+1}l_{m+1}},\dots,c_{i_nl_n})=\alpha,
\]
with $\underline{h}=\langle h_{i_1j_1},\dots,h_{i_mj_m}\rangle$. The relation (\ref{12caseii}) says that
\begin{equation}\label{Acaseii}
\forall\,\mathbf{y}\in\displaystyle{\prod_{i=1}^{d^{'}}}\big{[}Y^{*}_{p+i}\big{]}^r\,\,\,M\models\sigma_1(\mathbf{a},\mathbf{y}).
\end{equation}
Now by the equivalence of $\langle c_{i_{m+1}j_{m+1}},\dots,c_{i_nj_n}\rangle$ and $\langle c_{i_{m+1}l_{m+1}},\dots,c_{i_nl_n}\rangle$ together with Lemma \ref{combinatoriallemma}, we conclude that
\begin{equation}\label{Bcaseii}
\forall\,\mathbf{y}\in\displaystyle{\prod_{i=1}^{d^{'}}}\big{[}Y_{p+i}\big{]}^r\,\,\,M\models\sigma_2(\mathbf{a},\mathbf{y}).
\end{equation}
We get the following relation as a result of (\ref{Acaseii}) and (\ref{Bcaseii}):
\[
\forall\,\mathbf{y}\in\displaystyle{\prod_{i=1}^{d^{'}}}\big{[}Y_{p+i}\big{]}^r\,\,\,M\models\tau^{\mathrm{left}}(\mathbf{a},\mathbf{y})=
\tau^{\mathrm{right}}(\mathbf{a},\mathbf{y}).
\]
But
\[
\mathbf{b}\in\displaystyle{\prod_{i=1}^{d^{'}}}\big{[}Y_{p+i}\big{]}^r,
\]
so  $M\models\tau^{\mathrm{left}}(\mathbf{a},\mathbf{b})=
\tau^{\mathrm{right}}(\mathbf{a},\mathbf{b})$. This equals to say that $M\models\sigma(\mathbf{a},\mathbf{b})$. This completes the proof of Case II. Now we are in the position to say that the proof of Proposition \ref{proposition3} is finished.
\end{proof}

Now we are ready to complete the proof of the main theorem of this paper:

\begin{proof}[Proof of Theorem \ref{maintheorem} (i)]
It is enough to show that $T+\Sigma^{*}+\Sigma_1^{*}$ is consistent. Let $\Sigma^{'}$ be a finite part of $\Sigma^{*}+\Sigma_1^{*}$. Suppose $r,s$ are large enough positive integers so that for any $\sigma\in\Sigma^{'}$ and any $c_{ij}$ occurring in $\sigma$ we have $i\leq s$ and $j\leq r$. After the natural interpretation of all symbols of $\mathcal{L}^S$ in $M$, we have $M\models T_{skolem}$. We will show that there is $\mathbf{a}\in[F]^{r,s}$ such that for every $\sigma\in\Sigma^{'}$, $M\models\sigma(\mathbf{a})$. This would imply that $T+\Sigma^{*}+\Sigma_1^{*}$ is consistent. Note that $\Sigma^{*}($i$)=\Sigma_1^{*}($i$)$, $\Sigma^{*}($ii$)=\Sigma_1^{*}($ii$)$ and $\Sigma^{*}($iii$)=\Sigma_1^{*}($iii$)$ and we have shown in the proof of Theorem \ref{partial consistency} that if $\sigma$ is of types $\Sigma^{*}_1($i$),\Sigma^{*}_1($ii$)$ and $\Sigma^{*}_1($iii$)$, then for any $\mathbf{a}\in[F]^{r,s}$, $M\models\sigma(\mathbf{a})$. Now suppose that $B=\{\sigma_1,\dots,\sigma_p\}$ is the set of all sentences of $\Sigma^{'}$ of types $\Sigma^{*}($iv$),\Sigma_1^{*}($iv$)$. Set $\{\iota(\sigma_1),\dots,\iota(\sigma_p)\}=\{e_1,\dots,e_q\}$ such that $e_1<\dots<e_q$. Obviously $e_1>1$ and $e_q\leq s$ and also $[F]^{r,s}$ is 1-superlarge. By induction we shall show that there are sets $S_q\subset\dots\subset S_1\subset[F]^{r,s}$ such that for $1\leq k\leq q$, every $S_k$ is $e_k$-superlarge and if $\mathbf{a}\in S_k$, then $M\models\sigma(\mathbf{a})$, where $\sigma\in B$ and $\iota(\sigma)=e_k$. Put $S_0=[F]^{r,s},e_0=1$. Suppose we have constructed $S_{k-1}$ and we want to find $S_k$. Let
\[
B_k=\{\sigma\in B|\sigma\in\Sigma_1^{*}(\mathrm{iv}),\iota(\sigma)=e_k\},
\]
\[
B_k^*=\{\sigma\in B|\sigma\in\Sigma^{*}(\mathrm{iv}),\iota(\sigma)=e_k\}.
\]
If $B_k\neq\emptyset$, then by Proposition \ref{proposition2} there is an $e_k$-superlarge set $S_k^{(0)}\subset S_{k-1}$ such that
\[
\forall\,\sigma\in B_k\,\forall\,\mathbf{a}\in S_k^{(0)}\,\, M\models\sigma(\mathbf{a}).
\]
Note that if $B_k=\emptyset$, we do nothing and straightly turn to $B_k^*$. If  $B_k^*\neq\emptyset$ and $|B_k^*|=n_k$, then by a successive use of Proposition \ref{proposition3}, $n_k$ times, we get a finite nested sequence of $e_k$-superlarge sets:
\[
S_k^{(n_k)}\subset S_k^{(n_{k}-1)}\dots\subset S_k^{(1)}\subset S_k^{(0)}\subset S_{k-1}
\]
such that
\[
\forall\,\sigma\in B_k^*\,\forall\,\mathbf{a}\in S_k^{(n_k)}\,\, M\models\sigma(\mathbf{a}).
\]
Now we define $S_k$. If $B_k^*\neq\emptyset$, put $S_k=S_k^{(n_k)}$, otherwise put $S_k=S_k^{(0)}$. Therefore for all $\mathbf{a}\in S_k$ and all $\sigma\in B$ (and consequently all $\sigma\in\Sigma^{'}$) we have $M\models\sigma(\mathbf{a})$. This completes the proof.
\end{proof}

Now by a simple example in the below we show that there is a first-order theory $T$ such that $T$ has a strong singular-like model but none of its models (including its Keisler singular-like models) has an elementary end extension. So this justifies our assumption in Theorem \ref{maintheorem} that $T$ has a strongly inaccessible-like model.

Let $\kappa$ be a singular strong limit cardinal. Let $L=\{<,F,U\}$ be a first order language such that $F$ is a one-place function symbol and $U$ is a one-place relation symbol. Let $T$ be a first order theory in the language $L$ which says:

(i) $<$ is a linear order,

(ii) $U$ is a proper initial segment of the model,

(iii) $F$ maps $U$ cofinally into the model.

It is easy to see that the $\kappa$-like model $M=(\kappa,<)$ in which $<$ is interpreted as $\in$, can be expanded to a model $M=(\kappa,<,F,U)$ of $T$. Now if $M$ is a model of $T$ which has a proper elementary end extension $N$, then we must have $U^{M}=U^{N}$, but note that $F$ cannot cofinally map $U$ into both of $M$ and $N$. So $M$ cannot have a proper elementary end extension.

\bibliography{singular-like}
\bibliographystyle{plain}
\end{document}